\documentclass{amsart}

\usepackage{graphicx}
\usepackage{subcaption}
\usepackage{amsfonts}
\usepackage{color}
\usepackage{verbatim}

\usepackage[pdfborder={0 0 0},
  colorlinks=true]{hyperref}
\usepackage[nameinlink]{cleveref}
\usepackage[numbers,sort]{natbib}

\definecolor{newblue}{RGB}{94,89,144}
\definecolor{newblue2}{cmyk}{1,0.6,0,0.06}

\definecolor{sorange}{RGB}{203,75,22}
\definecolor{sblue}{RGB}{38,139,210}
\definecolor{smagenta}{RGB}{211,54,130}

\hypersetup{
  linkcolor  = sorange,
  citecolor  = sblue,
  urlcolor   = sblue,
}

\numberwithin{equation}{section}

\usepackage{amsthm}
\newtheorem{theorem}{Theorem}

\newtheorem{lemma}[theorem]{Lemma}

\newtheorem{problem}[theorem]{Problem}

\theoremstyle{remark}

\newtheorem{remark}[theorem]{Remark}

\crefname{section}{Sec.}{Sec.}
\crefname{table}{Tab.}{Tab.}
\crefname{figure}{Fig.}{Fig.}
\crefname{appendix}{App.}{App.}
\crefname{equation}{}{}
\crefname{theorem}{Thm.}{Thm.}
\crefname{proposition}{Prop.}{Prop.}
\crefname{lemma}{Lem.}{Lem.}
\crefname{corollary}{Cor.}{Cor.}
\crefname{problem}{Prob.}{Prob.}
\crefname{definition}{Def.}{Def.}
\crefname{deflisti}{Def.}{Def.}
\crefname{subdeflisti}{Def.}{Def.}
\crefname{remark}{Rem.}{Rem.}
\crefname{remlisti}{Rem.}{Rem.}
\crefname{note}{Note}{Note}
\crefname{assumption}{Ass.}{Ass.}
\crefname{example}{Ex.}{Ex.}
\crefname{exlisti}{Ex.}{Ex.}

\usepackage{amsaddr}

\title[Stable FEM for 3D undulatory locomotion]{A stable finite element method for low inertia undulatory locomotion in three dimensions}

\author[T. Ranner]{Thomas Ranner}
\address{School of Computing, University of Leeds, Leeds, UK. LS2 9JT}
\email{T.Ranner@leeds.ac.uk}
\thanks{\textbf{Funding.} This work was supported by a Leverhulme Trust Early Career Fellowship.}
\date{\today}

\newcommand{\abs}[1]{\left|#1\right|}
\renewcommand{\vec}[1]{\boldsymbol{#1}}
\newcommand{\id}{\mathbb{I}}
\newcommand{\K}{\mathbb{K}}
\newcommand{\R}{\mathbb{R}}
\newcommand{\rot}{\mathrm{rot}}
\newcommand{\ddt}{\frac{\mathrm{d}}{\mathrm{d} t}}

\ifpdf
\hypersetup{
  pdftitle={A stable finite element method for low inertia undulatory locomotion in three dimensions},
  pdfauthor={T. Ranner}
}
\fi

\begin{document}

\begin{abstract}
  We present and analyse a numerical method for understanding the low-inertia dynamics of an open, inextensible viscoelastic rod - a long and thin three dimensional object - representing the body of a long, thin microswimmer.
  Our model allows for both elastic and viscous, bending and twisting deformations and describes the evolution of the midline curve of the rod as well as an orthonormal frame which fully determines the rod's three dimensional geometry.
  The numerical method is based on using a combination of piecewise linear and piecewise constant finite element functions based on a novel rewriting of the model equations.
  We derive a stability estimate for the semi-discrete scheme and show that at the fully discrete level that we have good control over the length element and preserve the frame orthonormality conditions up to machine precision.
  Numerical experiments demonstrate both the good properties of the method as well as the applicability of the method for simulating undulatory locomotion in the low-inertia regime.
\end{abstract}

\maketitle


\noindent\textbf{Subject class.}{
  74S05, 74K10, 74D05, 65M60, 74L15, 92C10 }

\section{Introduction}

\subsection{Background}

The dynamics of active microswimmers has long captured the interest of physicists, mathematicians and engineers, not to mention researchers in various biological fields.
Undulation, passing bending waves down the body, is a common strategy used across many orders of magnitude \cite{Gazzola2014} ranging from bacteria \cite{Lig76,Lauga2016} and spermatozoa \cite{Fauci2006} to larger fish and whales \cite{Fish2006}.
It is especially common in the low inertia regime, where viscous forces of the surrounding media dominate over inertial forces and the scallop theorem prohibits self-propelled locomotion for time-reversal symmetric sequences of body postures \cite{Pur77,Tay52}.
For additional reading, we refer the reader to a large body
of excellent reviews on animal locomotion \cite{Gra68,Lauga2009,Cicconofri2019}, fluid dynamics at low Reynolds number \cite{Childress1981,Lighthill1975,Lig76,Brennen1977,Pur77,Yates1986,Fauci2006}, and the biophysics and biology of cell motility \cite{Berg2004,Holwill696,Jahn1972,Blum1979,Berg2000,Bray2001}.
There are many computational studies in this area which have focussed on solving a fully detailed three dimensional formulation capturing many aspects of the problem including the dynamics and its interaction with the body (e.g.\ \cite{Szigeti2014,Palyanov2016,Spagnolie2013,Elmi2017,Mujika2017,Simons2015}).
These large scale computational studies provide many interesting results giving insights on both complex fluids and animal locomotion (see, e.g., the review \cite{Lauga2009} and references therein).
An alternative is to choose a reduced, lower dimensional (in our case one dimensional) description  of the body and represent all forces on this reduced description.
This allows orders of magnitude faster simulations by considering only a one dimensional problem which  significantly reduces the numbers of computational degrees of freedom.

In this work we develop, analyse and show experiments using a novel computational approach to the evolution of an open viscoelastic rod, representing the body of a long, thin microswimmer - a three dimensional object with one axis much longer than the other two.
We assume the body is embedded in three dimensions which can undergo bending and twisting deformations.
The resulting problem can be described as a  system of one dimensional partial differential equations for a midline curve and an orthonormal frame which describes the conformation of cross sections to the midline.
The model is a natural generalisation of the locomotion model \cite{Guo2008} to three spatial dimensions.
The elastic terms we consider arise from the classical Kirchhoff-Love model for an elastic rod \cite{Kirchhoff1859} which are combined in a simple linear Voigt viscoelastic model \cite{Antman1996,LanLif75}.
Our model avoids considering shearing or extensional deformations from the full Cosserat or Timoshenko models.
In general this nonlinear system can not be solved analytically and requires the use of computer methods.
The reduction to a one dimensional object allows for significant reduction in the complexity of the resulting mathematical model in particular with respect to the computational effort required to solve the problem.

Our new method tackles three key challenges.
First when considering active locomotion one must be able to map directly between anatomical detail of the organism under consideration and the geometry defined by the model. For example, we should be able to clearly identify where the muscles are located and in which direction do they apply force.
Our approach to capture this results in equations for the midline of the rod coupled to equations for the orientation of the cross section of the rod (see the discussion in \cite{Langer1996} and \cref{sec:application} for more details). These equations must be carefully coupled to ensure that we have an accurate and robust representation of the geometry which allows us to apply the biological forces appropriately.
Second we have a moving geometry which is a priori unknown. Our scheme captures this with a parameterization which is equivalent to a moving mesh. As is typical in this type of problem we must ensure that the moving mesh does not become too distorted (see, for example, \cite{Elliott2016,Barrett2011a} and references therein).
Finally, in many biological systems various parameters and problem data may be unknown or poorly characterized. Any computational method should be both cheap enough to run so that parameterization studies may be run and also robust to input parameters so that a wide variety of behaviours can be observed.
See \cite{Stuart2010,kaipio2006statistical} for example.
We will demonstrate through analysis of our method and numerical experiments that we can address all three challenges.

Similar models to those considered in this paper arise in many areas of natural science and engineering.
For example, similar models have been considered for elastic ribbons and filaments \cite{bergou2008discrete,Bergou2010,Ftterer2012},
tangled hair and fibres \cite{Bertails2006,Ward2007,Daviet2011,Durville2005},
plants \cite{Goriely1998,Gerbode2012}, and
woven cloth \cite{Kaldor2008}.
A historical overview of the model used in this paper is given in \cite{dill1992kirchhoff}.

We are interested in the locomotion of low inertia microswimmers.
We explore the applicability of our method through the case study of the 1mm long nematode \textit{Caenorhabditis elegans}.
This animal has become a model organism studied by physicists, computer scientists and engineers partly due to its simple, undulating, periodic gait but also its experimental tractability and simple neuroanatomy \cite{Cohen2014}.
In its natural environment, \textit{C.\ elegans} grows mostly in rotten vegetation: a highly complex, heterogeneous three dimensional environment. It is cultured on the surface of agar gels for extensive use in laboratories.
On agar, \textit{C.\ elegans} move by propagating bending waves from head to tail to generate forward thrust \cite{GraLis64,Cro70,PieCheMun08,BerBoyTas09,FanWyaXie10,Lebois2012} where the bending arises from body wall muscles which line sides of the body.
A two dimensional version of our viscoelastic model \cite{Guo2008}, developed originally for capturing worms or snakes moving across land, has previously been applied to \textit{C. elegans} locomotion both on surfaces and in a range of both Newtonian and non-Newtonian media \cite{SznPurKra10,Sznitman2010,FanWyaXie10,CohRan17-pp,Denham2018}.
However, previous formulations have either linearized the equations, which means although postures are recovered trajectories are not, or viscosity is neglected, which limits the applicability of the model in less resistive environments.
Recent experiments \cite{Bilbao2018,Shaw2018} have shown that \textit{C. elegans} achieves a different gait in three dimensions.
Here we demonstrate that our computational tool is capable of accurately and robustly capturing such behaviours.
The simulation results are meant to be indicative of the capabilities of the method rather than demonstrating any properties of the underlying model which is left to future work.
The numerical method used in this section builds on the unpublished work \cite{CohRan17-pp} (see also \cite{Denham2018}).

\subsection{Model}
\label{sec:intro-model}

In our model, the rod is described through a smooth, time-dependent parameterization of the midline $\vec{x} \colon [0,1] \times [0,T] \to \R^3$ and an orthonormal frame $\vec{e}^0, \vec{e}^1, \vec{e}^2 \colon [0,1] \times [0,T] \to \R^3$ up to some final time $T>0$.
We call $u \in [0,1]$ the material parameter.
Since we do not allow shear deformations, we set $\vec{e}^0 \equiv \vec\tau := \vec{x}_u / \abs{ \vec{x}_u }$ the unit tangent to the midline. The other two coordinates $\vec{e}^1$ and $\vec{e}^2$ form an orthonormal basis of the cross section of the rod to the midline. See \cref{fig:geometry}. We can form the skew system:
\begin{align*}
  \frac{1}{\abs{\vec{x}_u}} \vec{e}^{0}_u
  = \alpha \vec{e}^{1} + \beta \vec{e}^{2}, \qquad
  \frac{1}{\abs{\vec{x}_u}} \vec{e}^{1}_u
  = - \alpha \vec{e}^{0} + \gamma \vec{e}^{2}, \qquad
  \frac{1}{\abs{\vec{x}_u}} \vec{e}^{2}_u
  = - \beta \vec{e}^{0} - \gamma \vec{e}^{1},
\end{align*}
where $\alpha$ and $\beta$ are smooth fields denoting the curvature of the midline in directions $\vec{e}^{1}$ and $\vec{e}^{2}$ and $\gamma$ is a smooth field denoting the twist of the orthonormal frame about the midline.
At the core of the model is a moment which is the sum of elastic and viscous contributions.
The elastic terms are proportional to differences between the fields $\alpha, \beta, \gamma$ and some desired values $\alpha^0, \beta^0, \gamma^0$.
The viscous contribution is proportional to the time derivatives $\alpha_t, \beta_t, \gamma_t$.

Inertial terms are ignored and external forces are simply modelled as linear drag terms \cite{Keller1976}.
This model can be seen as a quasi-static or low inertia approximation of the full rod dynamics.
Although this very simple approach to modelling external forces for undulatory locomotion has been used previously in undulatory locomotion models (see the references above),
a more accurate model would include non-local terms \cite{Keller1976,Lighthill1975} or capture the external fluid using dynamic equations  (e.g.\ \cite{Cortez2018,schoeller2019methods,MontenegroJohnson2016,Lim2012}).
Any model for the surrounding fluid must be coupled to a model of the body.
The focus in this work is to accurately capture the body mechanics combined with a simple model of surrounding fluid.
Study of a more detailed model of the environment coupled to the model of body mechanics presented here is left to future work.

Our model can be seen as a simplification of the model presented in \cite{Lang2010,Lang2012} or a three dimensional version of \cite{Guo2008}.
Full details of the model are given in \cref{sec:model}.

\subsection{Computational method}
\label{sec:intro-numerics}

Our approach is to discretize an appropriate formulation of the continuous equations directly.
This allows us to use the structure of the equations to recover a robust numerical scheme through careful discretization choices.
The discretization extends the approach of \cite{Lin2004} to the case of non-constant twist and open curves.

Key to our numerical method is a well suited formulation of the continuous partial differential equation system.
We start from the balance laws for linear and angular momentum and the linear viscoelastic constitutive law. This is combined with geometric constraints so that the solution variables are the position of the midline, line tension (a Lagrange multiplier for enforcing the length constraint), the curvature of the midline, the twist and angular velocity of the frame and two auxiliary variables which describe the bending and twisting moments.
The continuous system of equations is discretized in space by a mixed finite element method where we use a mix of piecewise linear and piecewise constant approximation spaces.
We use a Lagrange multiplier to enforce inextensibility (i.e.\ that the length of the curve is locally fixed) in a similar approach to \cite{Bar13}.
Finally, we discretize in time using a semi-implicit method which results in a linear system of equations to solve at each time step, inspired by the approach in \cite{Dzi90,DziKuwSch02,Lin2004} (see also \cite{Barrett2011a}), together with an update formula for the frame.
The use of lower order finite element approximations, along with mass lumping \cite{Tho84}, allows us to derive identities for various geometric quantities at the mesh vertices.
The frame is updated using a Rodrigues formula where the rotation is specified by the frame's angular moment which can be derived from the solution variables.
We use the angular momentum and Rodrigues formula as an alternative to using Euler angles \cite{Schwab2006} or quaternions \cite{Lang2010,Lang2012}.
Although direct use of the Rodrigues formula is not recommended for numerical applications in general \cite{Bauchau2003}, we will show that our formulation avoids problems for any angles.
More details of the numerical scheme are given in \cref{sec:method}.

We will demonstrate three key properties of our scheme:
\begin{itemize}
\item a semi-discrete stability result (coupled with computational evidence of fully discrete stability) which shows that we recover a discrete Lyapunov functional for our scheme;
\item control over the length element which ensures that vertices in the moving mesh do not collide;
\item preservation of the frame orthogonality conditions to ensure that the frame really does remain orthonormal over time.
\end{itemize}
We see these three properties as allowing our method to provide a computational tool both for the understanding of viscoelastic rods and also for domain experts working on undulatory locomotion in the low inertia regime.
Further work is required to link this model for undulatory locomotion with a more detailed model of the surrounding fluid which would give a more accurate model and allow a greater variety of behaviours to be captured. Where such links are to be used a balance must be struck between the computational efficiency of a one dimensional approach against the greater detail provided by a fully three dimensional model.

In our previous unpublished work \cite{CohRan17-pp}, we presented a similar scheme (also used in \cite{Denham2018}).
This paper builds on that work extending the scheme to three dimensions including twisting as well as bending contributions and generalising the material law to include viscous terms.

There are several very successful methods from discrete differential geometry which solve either the problem we consider or a generalisation.
In this approach the rod is first discretized and then equations are derived to evolve those discrete quantities.
We mention in particular the approach of \cite{bergou2008discrete} (extended to viscous threads in \cite{Bergou2010,Audoly2013}) which uses a discrete set of vertices (in our notations a piecewise linear curve) with a material frame on each edge between the vertices (piecewise constant in our notation) to represent a rod which was stored using a reduced curve-angle formulation \cite{Langer1996}.
The approach uses curvature and twist defined as integrated quantities based at vertices to define a Kirchhoff-Love energy and then applies discrete parallel transport and variation of holonomy to derive the update equations.
This work was generalized to Cosserat rods by \cite{Gazzola2018} who revert away from the reduced curvature-angle formulation storing the full frame.
The super-helix and super-clothoid approaches \cite{Bertails2006,Casati2013} use a piecewise constant or piecewise linear approximation of generalized curvature of a rod and then recover the geometry of the rod (position of midline and frame) using analytic expressions.
The scheme results in a smooth curve with well defined curvatures and twist.

In the context of these schemes our model can be seen as using a set of vertices with a frame at each vertex to define the discrete geometry of the rod.
The equations are straight discretization of the continuous scheme using a finite element approach and discrete equations to define curvature and relate twist with tangential angular velocity.
Our choice to discretize continuum equations allows us to use standard finite element tools to both analyse and implement the method in a single unified framework.
This approach is well suited to the parabolic nature of the equations we consider.
It is our particular choice of both geometric discretization and direct discretization of the equations that allows us to demonstrate the properties of our scheme.

\subsection{Outline}

In \cref{sec:model}, we present the continuous model we use and the discretization is shown in \cref{sec:method}. Numerical experiments to demonstrate the efficacy of the method are shown in \cref{sec:results}.
The restriction of our scheme to a two dimensional problem is given in \cref{sec:2d-problem}.

\section{Governing equations}
\label{sec:model}

\subsection{Geometry}

We consider a smooth, inextensible, unshearable rod embedded in $\R^3$ over a time interval $[0,T]$ for some $0 < T < \infty$.
The rod can be described by a (non-arc length) parameterization of the centre line $\vec{x} \colon [0,1] \times [0,T] \to \mathbb{R}^3$ and an oriented frame of reference $\vec{Q} \colon [0,1] \times [0,T] \to SO(3)$.
For a discussion of rod representations see \cite{Langer1996}.
We will write derivatives with respect to the first coordinate, the material coordinate $u$, and the second coordinate, time $t$, with subscripts $(\cdot)_u$ and $(\cdot)_t$, respectively.
Our assumptions imply $\abs{ \vec{x}_u }_t = 0$ so that the length of the midline is fixed. We call the length of the midline curve $L$.

Rather than use the tensor $\vec{Q}$, we will use the equivalent orthonormal triad of unit vectors $\vec{Q} = \{ \vec{e}^0, \vec{e}^1, \vec{e}^2 \}$.
We will use the convention for an unshearable rod that $\vec{e}^0 \equiv \vec\tau$ the unit tangent to the centre line given by $\vec\tau = \vec{x}_u / | \vec{x}_u |$. See \cref{fig:geometry} for an example rod conformation.

\begin{figure}[tbhp]
  \centering
  \includegraphics[width=0.5\textwidth]{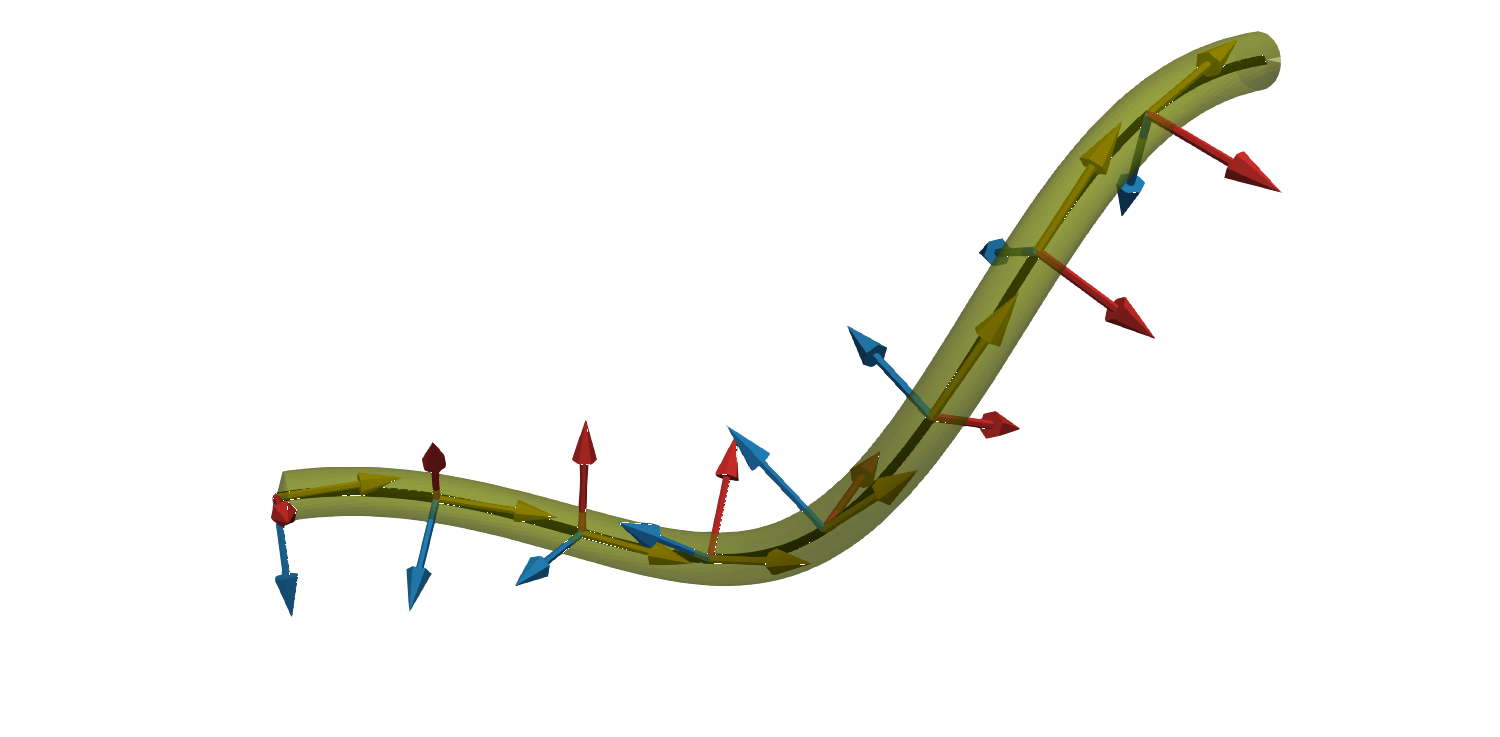}
  \caption{An illustration of a conformation of a rod. The image shows a (green) shaded three dimensional region which can be parameterized by a midline curve $\vec{x}$ (black) and an orthogonal triad of vectors $\vec{e}^{0}$ (yellow), $\vec{e}^{1}$ (red), and $\vec{e}^{2}$ (blue).
  The fields $\alpha$ and $\beta$ represent the variation (along the curve) of the tangent vector $\tau = \vec{e}^{0}$ and $\gamma$ represents the rotation of the pair $\vec{e}^{1}, \vec{e}^{2}$ about the tangent vector $\tau = \vec{e}^{0}$.}
  \label{fig:geometry}
\end{figure}

We can recover the generalized curvature (sometimes called the Darboux vector) $\vec\Omega \colon [0,1] \times [0,T] \to \R^3$ which satisfies
\begin{equation}
  \label{eq:darboux}
    \frac{1}{\abs{\vec{x}_u}} \vec\tau_u
    = \vec\Omega \times \vec\tau, \qquad
    \frac{1}{\abs{\vec{x}_u}} \vec{e}^{1}_u
    = \vec\Omega \times \vec{e}^{1}, \qquad
    \frac{1}{\abs{\vec{x}_u}} \vec{e}^{2}_u
    = \vec\Omega \times \vec{e}^{2}.
\end{equation}
We decompose $\vec\Omega = \vec\tau \times \vec{\kappa} + \gamma \vec\tau$. We call $\vec{\kappa} = \vec\tau_u/\abs{\vec{x}_u}$ the vector curvature with decomposition $\vec{\kappa} = \alpha \vec{e}_1 + \beta \vec{e}_2$ for fields $\alpha = \vec{\kappa} \cdot \vec{e}_1, \beta = \vec{\kappa} \cdot \vec{e}_2$ and call $\gamma$ the twist.
Further, if the orthonormal triad transforms smoothly in time, then we recover the angular velocity of the frame $\vec\omega$ given by
\begin{equation}
  \label{eq:angular-velocity}
  \vec\tau_t = \vec\omega \times \vec\tau, \qquad
  \vec{e}^{1}_{t} = \vec\omega \times \vec{e}^{1}, \qquad
  \vec{e}^{2}_{t} = \vec\omega \times \vec{e}^{2}.
\end{equation}
Again, we decompose $\vec\omega = \vec\tau \times \vec\tau_t + m \vec\tau$.
The field $m$ denotes the angular velocity of the frame about the tangent vector field which must be computed separately to the centre line velocity.

There is a geometric relation between the strain vector $\vec\Omega$ and the angular velocity $\vec\omega$ \cite{Wolgemuth2004}.
Since the $u$ and $t$ derivatives commute, using the inextensibility of the parameterization (i.e.\ $\abs{\vec{x}_u}_t=0$), we can compute that
\begin{equation}
  \label{eq:deriv-1}
  \frac{1}{\abs{\vec{x}_u}} \vec\omega_u - \vec{\Omega}_t
  = \vec\omega \times \Omega.
\end{equation}
Next we subtract the time derivative of $\gamma = \vec\Omega \cdot \vec\tau$ from the scaled $u$ derivative (i.e. arc-length derivative) of $m = \vec\omega \cdot \vec\tau$ to see
  \[
    \frac{m_u}{\abs{\vec{x}_u}} - \gamma_t
    = \left( \frac{\vec{\omega}_u}{\abs{\vec{x}_u}} - \vec{\Omega}_t \right) \cdot \vec{\tau} + \vec{\omega} \cdot \frac{\vec{\tau}_u}{\abs{\vec{x}_u}} + \vec{\Omega} \cdot \vec{\tau}_t.
  \]
  Applying \cref{eq:deriv-1,eq:angular-velocity}, the definition of $\vec{\kappa}$ and rearranging the vector triple product gives
  \[
 \frac{m_u}{\abs{\vec{x}_u}} - \gamma_t
    = \vec\tau \cdot ( \vec\omega \times \vec\Omega ) + \vec\omega \cdot \vec\kappa + \vec\tau \cdot ( \vec\Omega \times \vec\omega )
 = \vec{\omega} \cdot \vec\kappa.
\]
Then applying the definition of $\vec\omega$ and rearranging the resulting vector triple product, we see that
\begin{equation}
  \label{eq:equation-gamma}
  \frac{m_u}{\abs{\vec{x}_u}} = \gamma_t - \frac{\vec{x}_{tu}}{\abs{\vec{x}_u}} \cdot \vec{\tau} \times \vec{\kappa}.
\end{equation}
This final equation is important for providing a closing relation between the angular velocity and twist of the frame.

\subsection{Model derivation}

For an inextensible, unshearable rod we can write down the conservation of linear and angular momentum as \cite{LanLif75}:
\begin{equation}
  \label{eq:balance}
  \vec{K} + \frac{1}{\abs{\vec{x}_u}} \vec{F}_u = \vec{0}, \qquad
  \vec{K}^\rot + \vec\tau \times \vec{F} + \frac{1}{\abs{\vec{x}_u}} \vec{M}_u = \vec{0},
\end{equation}
where $\vec{K}$ is the external force, $\vec{F}$ is the internal force resultant, $\vec{K}^\rot$ is the external moment and $\vec{M}$ is the internal moment.

We assume that the rod is viscoelastic with preferred curvatures and preferred twist so that the internal moment is given by a linear Voigt model:
\begin{equation}
  \label{eq:moment}
  \vec{M} = \vec\tau \times \Bigl\{ A \bigl( ( \alpha - \alpha^0 ) \vec{e}^{1} + ( \beta - \beta^0 ) \vec{e}^{2} \bigr)
  + B ( \alpha_t \vec{e}^{1} + \beta_t \vec{e}^{2} ) \Bigr\}
  + C ( \gamma - \gamma^0 ) \vec\tau + D \gamma_t \vec\tau.
\end{equation}
Here $\alpha^0, \beta^0$ and $\gamma^0$ are given fields which we allow to depend on the parameter $u$ and time $t$.
We call $\alpha^0$ and $\beta^0$ preferred curvatures and $\gamma^0$ a preferred twist.
The material parameters, which we allow to depend on the parameter $u$ but not time $t$, are $A$ the bending modulus, $C$ the twisting modulus, $B$ the bending viscosity and $D$ the twisting viscosity.
Note that the material parameters will depend on the precise geometry of the cross section of the rod \cite{Lang2012}.
We assume that $A = A(u) \ge A_0 > 0$, $B = B(u) \ge 0$, $C = C(u) > C_0 > 0$ and $D = D(u) \ge 0$.
We introduce the variables $\vec{y}$, the bending moment, and $z$ the twisting moment given by
\begin{subequations}
  \label{eq:moment-split}
  \begin{align}
    \vec{y} & = A \bigl( ( \alpha - \alpha^0 ) \vec{e}^{1} + ( \beta - \beta^0 ) \vec{e}^{2} \bigr)
              + B ( \alpha_t \vec{e}^{1} + \beta_t \vec{e}^{2} ) \\
    z & = C ( \gamma - \gamma_0 ) + D \gamma_t,
  \end{align}
\end{subequations}
so that the moment is given by
\begin{equation}
  \label{eq:moment-simple}
  \vec{M} = \vec\tau \times \vec{y} + z \vec{\tau}.
\end{equation}
We assume that the tangential forces, $p \vec\tau$, act as a Lagrange multiplier to enforce the inextensibility constraint:
\begin{equation}
  \label{eq:inextensibility}
  \abs{ \vec{x}_u }_t = \vec\tau \cdot \vec{x}_{tu} = 0.
\end{equation}
We assume a linear drag response from the environment onto the rod by
\begin{equation}
  \label{eq:drag}
  \vec{K} = \K \vec{x}_t, \qquad
  \vec{K}^\rot = - K^\rot m \vec{\tau},
\end{equation}
with strictly positive definite matrix $\K$  and strictly positive scalar coefficient $K^\rot$.
Our model of drag is inspired by resistive force theory \cite{Keller1976}.

We combine the above model components in a way that is well suited to numerical computation.
We start by calculating from \cref{eq:moment-simple} that
\begin{align*}
  \frac{1}{\abs{\vec{x}_u}} \vec{M}_u =
  \vec\kappa \times \vec{y} + \vec\tau \times \frac{1}{\abs{\vec{x}_u}} \vec{y}_u
  + \frac{1}{\abs{\vec{x}_u}} \vec{z}_u \vec\tau + z \vec\kappa.
\end{align*}
Noting that $\vec\tau \cdot \vec\kappa = \vec\tau \cdot \vec{y} = \vec\tau \times ( \vec\kappa \times \vec{y} ) = 0$, we infer
\begin{subequations}
  \label{eq:deriv-2}
  \begin{align}
    \label{eq:deriv-2a}
    \vec\tau \cdot
    \frac{1}{\abs{\vec{x}_u}} \vec{M}_u
    & = \vec\tau \cdot ( \vec\kappa \times \vec{y} ) + \frac{1}{\abs{\vec{x}_u}} \vec{z}_u \vec\tau \\
    \label{eq:deriv-2b}
    \vec\tau \times
    \frac{1}{\abs{\vec{x}_u}} \vec{M}_u
    & = ( \id - \vec\tau \otimes \vec\tau ) \frac{\vec{y}_u}{\abs{\vec{x}_u}}
      + z \vec\tau \times \vec\kappa,
  \end{align}
\end{subequations}
Here $\id$ is the $3 \times 3$ identity matrix and $\otimes$ is the outer product given by $( \vec{a} \otimes \vec{b} )_{ij} = \vec{a}_i \vec{b}_j$ for $i,j=1,2,3$, $\vec{a}, \vec{b} \in \R^3$ which satisfies $\vec\tau \times ( \vec\tau \times \vec{a} ) = ( \id - \vec\tau \otimes \vec\tau) \vec{a}$ for all $a \in \R^3$.

We decompose the internal force resultant into tangential and normal components by $\vec{F} = p \vec{\tau} + \vec{f}$.
We call $p$ the pressure and $\vec{f}$ the normal force resultant.
We then use the force decomposition and \cref{eq:drag} in the linear momentum balance \cref{eq:balance}:
\begin{equation}
  \label{eq:deriv-3a}
  \K \vec{x}_t + \frac{1}{\abs{\vec{x}_u}} ( p \vec\tau )_u + \frac{1}{\abs{\vec{x}_u}} \vec{f}_u = \vec{0}.
\end{equation}
From taking cross product of $\vec\tau$ with the angular momentum balance \cref{eq:balance} and then the arc-length derivative (the scaled $u$ derivative), we infer:
\begin{align}
  \label{eq:deriv-3b}
  \frac{1}{\abs{\vec{x}_u}} \vec{f}_u + \frac{1}{\abs{\vec{x}_u}} \left( \frac{1}{\abs{\vec{x}_u}} \vec\tau \times \vec{M}_u \right)_u & = \vec{0}.
\end{align}
Next, we take the scalar product of $\vec\tau$ with the angular momentum balance \cref{eq:balance} using \cref{eq:drag} to see:
\begin{align}
  \label{eq:deriv-3c}
  -K^\rot m + \frac{1}{\abs{\vec{x}_u}} \vec{M}_u \cdot \vec{\tau} & = 0.
\end{align}
Finally we apply the expressions for derivatives of $\vec{M}$ from \cref{eq:deriv-2} in \cref{eq:deriv-3a}, \cref{eq:deriv-3b} and \cref{eq:deriv-3c} and combine with \cref{eq:inextensibility} to give our model:
\begin{subequations}
  \label{eq:model}
  \begin{align}
    \label{eq:model-x}
    \K \vec{x}_t + \frac{1}{| \vec{x}_u |} \bigl( p \vec\tau \bigr)_u
    + \frac{1}{|\vec{x}_u|} \bigl( ( \id - \vec\tau \otimes \vec\tau) \frac{ \vec{y}_u }{| \vec{x}_u |}  + z \vec\tau \times \vec{\kappa} \bigr)_u & = \vec{0}
    \\
    \label{eq:model-m}
    - K^\rot m
    + \frac{ z_u }{| \vec{x}_u |} + \vec{y} \cdot ( \vec\tau \times \vec{\kappa} ) & = 0
    \\
    \label{eq:model-p}
    \vec\tau \cdot \vec{x}_{tu} & = 0.
  \end{align}
  For boundary conditions we assume that each end of the rod is free so we enforce zero force and zero moment at $u=0,1$:
  \begin{align}
    \label{eq:model-bc0}
  p \vec\tau
    + ( \id - \vec\tau \otimes \vec\tau) \frac{\vec{y}_u}{| \vec{x}_u |}  + z \vec\tau \times \vec{\kappa} & = \vec{0} && \mbox{ at } u=0,1\\
    \label{eq:model-bc1}
  \vec\tau \times \vec{y} + z \vec\tau & = \vec{0} && \mbox{ at } u=0,1.
  \end{align}
\end{subequations}

\begin{remark}
  The system of partial differential equations \cref{eq:model} can be seen as a nonlinear fourth order parabolic equation for the parameterization $x$, subject to the nonlinear inextensibility constraint, coupled to a nonlinear second order parabolic equation for the twist $\gamma$.
\end{remark}

\subsection{Weak form}

We will write down a weak form which we will use for our finite element method in the next section.
When writing down the weak formulation, we combine equations for the conservation laws \cref{eq:model-x} and \cref{eq:model-m}, the constitutive laws \cref{eq:moment-split}, the geometric relation \cref{eq:equation-gamma}, and the inextensibility constraint \cref{eq:model-p}.
When writing down the constitute equation we add the Laplace-Beltrami identity for $\vec{\kappa}$ the vector curvature:
\begin{equation}
  \label{eq:model-w}
  \vec{\kappa} = \frac{1}{\abs{\vec{x}_u}} \vec{\tau}_u = \frac{1}{\abs{\vec{x}_u}} \left( \frac{ \vec{x}_u }{ \abs{ \vec{x}_u }} \right)_u.
\end{equation}
We must also impose boundary conditions for $\vec{\kappa}$, since the curvature is not well defined at $u=0,1$, which we set to be equal to the prescribed curvatures here:
\begin{equation}
  \label{eq:model-w-b}
  \vec{\kappa} = \alpha^0 \vec{e}^{1} + \beta^0 \vec{e}^{2} \qquad \mbox{ at } u = 0,1.
\end{equation}
We also note that for the bending viscosity terms we can write
\[
  \alpha_t \vec{e}^{1} + \beta_t \vec{e}^{2} = ( \id - \vec\tau \otimes \vec\tau ) \vec{\kappa}_t - m \vec\tau \times \vec{\kappa}.
\]
We derive the weak form of the problem by multiplying by appropriate test functions and integrating over the centre line.

We let $Q = L^2( 0, 1 )$ denote the space of square integrable functions on $(0,1)$,
$V = H^1( 0, 1 )$ the Sobolev space of functions in $L^2(0,1)$ with a weak derivative in $L^2(0,1)$
and $V_0$ the space of functions in $V$ with zero trace \cite{Evans2010}.
Unless otherwise stated, integrals are with respect to the measure $\mathrm{d} u$.

\begin{problem}
  Given preferred curvatures, $\alpha^0, \beta^0$, a preferred twist, $\gamma^0$, and initial conditions for the parameterization $\vec{x}^0$ and frame $\vec{e}^{1,0}, \vec{e}^{2,0}$,
  find $\vec{x} , \vec{y}, \vec{\kappa} \colon [0,1] \times [0,T) \to \R^3$ (with the conditions \cref{eq:model-bc0,eq:model-bc1,eq:model-w-b} at the boundaries),
$m, z, \gamma \colon [0,1] \times [0,T) \to \R$, and
$\vec{e}^{1}, \vec{e}^{2} \colon [0,1] \times [0,T) \to \R^3$ such that, for almost every $t \in (0,T)$:
\begin{subequations}
  \label{eq:weak}
\begin{align}
  \label{eq:weak-x}
  \int_0^1 \K \vec{x}_t \cdot \vec{\phi} | \vec{x}_u |
  - \int_0^1 p \vec\tau \cdot \vec{\phi}_u
  - \int_0^1 \bigl( ( \id - \vec\tau \otimes \vec\tau ) \frac{1}{| \vec{x}_u |} \vec{y}_u  + z \vec\tau \times \vec{\kappa} \bigr) \cdot \vec{\phi}_u
  & = 0 \\
\label{eq:weak-y}
  \int_0^1 \bigl( \vec{y} - A ( \vec{\kappa} - \alpha^0 \vec{e}_{1} - \beta^0 \vec{e}_{2} )
  - B \bigl( ( \id - {\vec\tau} \otimes {\vec\tau} ) \vec{\kappa}_{t} - m {\vec\tau} \times \vec{\kappa} \bigr)
  \bigr) \abs{ \vec{x}_{u} }
   & = 0 \\
  \label{eq:weak-w}
  \int_0^1 \vec{\kappa} \cdot \vec\psi | \vec{x}_u | + \frac{\vec{x}_u}{ | \vec{x}_u | } \cdot \vec\psi_u & = 0
\end{align}
for all $\vec\phi \in V^3, \vec\psi \in V_0^3$,
\begin{align}
  \label{eq:weak-m}
  \int_0^1 - K^\rot m v | \vec{x}_u | + \int_0^1 \vec{y} \cdot ( \vec\tau \times \vec{\kappa} ) v | \vec{x}_u |
  - \int_0^1 z v_u & = 0, \\
  \label{eq:weak-z}
  \int_0^1 ( z - C ( \gamma - \gamma^0 ) - D \gamma_t ) q | \vec{x}_u | & = 0, \\
  \label{eq:weak-gamma}
  \int_0^1 \gamma_t q | \vec{x}_u |
  - \int_0^1 m_u q
  + \int_0^1 \vec\tau \times \vec{\kappa} \cdot \vec{x}_{tu} q & = 0
\end{align}
for all $q \in Q$ and $v \in V$,
\begin{align}
  \label{eq:weak-p}
  \int_0^1 q \vec\tau \cdot \vec{x}_{tu} & = 0,
\end{align}
for all $q \in Q$,
and
\begin{align}
  \int_0^1 \bigl( \vec{e}^{j}_{t} - \bigl( \vec\tau \times \vec\tau_t + m \vec\tau \bigr) \times \vec{e}^{j} \bigr) \cdot \vec{\phi} \abs{ \vec{x}_{u} } = 0, \quad \mbox{ for } j = 1,2,
\end{align}
for all $\vec{\phi} \in V^3$,
subject to the initial conditions
\begin{equation}
  \vec{x}( \cdot, 0) = \vec{x}^0, \qquad
  \vec{e}^1( \cdot, 0 ) = \vec{e}^{1,0}, \qquad
  \vec{e}^2( \cdot, 0 ) = \vec{e}^{2,0},
\end{equation}
and initial equations
\begin{align}
  \int_0^1 \vec{\kappa}(\cdot, 0) \cdot \vec\psi | \vec{x}^0_u |
  + \frac{\vec{x}^0_u}{ | \vec{x}^0_u | } \cdot \vec\psi_u
  & = 0 && \mbox{ for all } \vec\psi \in V_0^3 \\
  \int_0^1 ( \gamma(\cdot, 0) - \frac{\vec{e}^{1,0}_{u}}{\abs{\vec{x}^0_u}} \cdot \vec{e}^{2,0} ) v \abs{\vec{x}^0_{u}}
  & = 0 && \mbox{ for all } v \in V.
\end{align}
\end{subequations}
\end{problem}

\section{Numerical method}
\label{sec:method}

\subsection{Finite element spaces}
We take a partition of $[0,1]$ by $N$ points $u_1 = 0 < u_2 < \ldots < u_{N} = 1$.
We call $\{ u_1, \ldots, u_N \}$ the mesh.
We will use a combination of piecewise linear and piecewise constant functions. We introduce the spaces:
\begin{align*}
  V_h & := \{ v_h \in C([0,1]) : v_h|_{[u_i,u_{i+1}]} \text{ is affine, for } i = 1, \ldots, N-1 \} \\
  Q_h & := \{ q_h \in L^2(0,1) : q_h|_{[u_i,u_{i+1}]} \text{ is constant, for } i = 1, \ldots, N-1 \}.
\end{align*}
We will denote by $V_{h,0}$ the space of finite element functions in $v_h \in V_h$ such that $v_h(0) = v_h(1) = 0$.
We will use the subscript $h$, defined to be the maximum mesh spacing, to denote discrete quantities. Temporal and spatial derivatives of discrete functions will be denoted by subscript $u$ and $t$ with a comma separating from the subscript $h$: e.g.\ $\eta_{h,u}$ denote the spatial derivative of $\eta_h$.

For a discrete parameterization $\vec{x}_h \in V_h^3$, we introduce two different tangent vector fields. First, $\vec{\tau}_h \in Q_h^3$ as the piecewise constant normalized derivative of $\vec{x}_h$:
\begin{equation}
  \label{eq:tauh-1}
  \vec{\tau}_h = \frac{\vec{x}_{h,u}}{ |\vec{x}_{h,u}| }.
\end{equation}
We will also require a piecewise linear approximation of $\vec{\tau}_h$ written $\tilde{\vec{\tau}}_h \in V_h^3$ with vertex values given by
\begin{align}
  \label{eq:tauh-2}
  \tilde{\vec{\tau}}_h( u_j, \cdot )
  & = \frac{ \vec{\tau}_{h}( u_i^-, \cdot) + \vec{\tau}_h( u_i^+, \cdot ) }{| \vec{\tau}_{h}( u_i^-, \cdot) + \vec{\tau}_h( u_i^+, \cdot ) |} \quad \mbox{ for } i = 1, \ldots, N,
\end{align}
where $\vec\tau_h( u_i^\pm, \cdot )$ is $\vec{\tau}_h$ evaluated on the left (or right) element to the vertex $u_i$.

We will apply mass lumping \cite{Tho84} using the notations:
\[
  ( f )_h := I_h( f ) \quad \mbox{ and } \quad \abs{ f }_{h} := \abs{ I_h( f ) } \mbox{ for } f \in C([0,1]),
\]
where $I_h$ is the Lagrangian interpolation operator $C([0,1]) \to V_h$.

Finally we denote by $V_{h,0}^3 + \vec{\kappa}_{b}(\cdot,t)$ the space of finite element functions $\vec{v}_h \in V_h^3$ which match the boundary conditions for $\vec{\kappa}_h$:
\[
  \vec{v}_{h} |_{u=0,1} = \vec{\kappa}_b := \alpha^0 \vec{e}^{1}_{h} + \beta^0 \vec{e}^{2}_{h},
\]
where $\vec{e}^1_h, \vec{e}^2_h \in V_h^3$ will denote components of the orthonormal frame that we will solve for as part of the method.
The space $V_{h,0}^3 + \vec{\kappa}_b$ will in general be time dependent.

\subsection{Semi-discrete problem}

We directly discretize the weak form \cref{eq:weak}.
The choice of piecewise linear or piecewise constant approximation spaces for the different functions is determined by the properties we will show in \cref{lem:stability}.

At this stage of discretization the choices are between which discrete function spaces each solution variable should live in and how to implement boundary conditions.
We choose piecewise linear approximations of position $\vec{x}$, bending moment $\vec{y}$, curvature $\vec{\kappa}$, angular momentum $m$ and frame $\vec{e}_{1}$ and $\vec{e}_2$ and piecewise constant approximations of twisting moment $z$, twist $\gamma$ and the Lagrange multiplier $p$.
We choose to enforce boundary conditions for the bending moment $\vec{y}$ and curvature $\vec{\kappa}$ in the function spaces but boundary conditions for the twisting moment $z$ and twist $\gamma$ arise as natural boundary conditions.
Imposing natural conditions means that the boundary conditions are only achieved exactly in the limit of small mesh spacing (numerical confirmation not shown).
We will see our choices naturally lead to the key properties of our scheme.
A summary of discretization choices in given in \cref{tab:variables}.

\begin{table}[tb]
  \centering
  \begin{tabular}[tbh]{ccc}
    \hline
    Variable & label & discrete space \\
    \hline
    Position & $\vec{x}_h$ & $V_h^3$ \\
    Lagrange multiplier & $p_h$ & $Q_h$ \\
    Vector curvature & $\vec{\kappa}_h$ & $V_{h,0}^3 + \kappa_b$ \\
    Twist & $\gamma_h$ & $Q_h$ \\
    Tangential angular velocity & $m_h$ & $V_h$ \\
    Normal moment & $\vec{y}_h$ & $V_{h,0}^3$ \\
    Tangential moment & $z_h$ & $Q_h$ \\
    Frame vectors ($j=1,2$) & $\vec{e}^{j}_{h}$ & $V_{h,0}^3$ \\
    \hline
  \end{tabular}
  \caption{Summary of discretization choices for terms in model. Recall the $V_h$ is the space of piecewise linear functions and $Q_h$ is the space of piecewise constant functions.}
  \label{tab:variables}
\end{table}

\begin{problem}
  Given preferred curvatures $\alpha^0, \beta^0$, a preferred twist $\gamma^0$,
  and initial conditions for the parameterization $\vec{x}_h^0$ and frame $\vec{e}_{h}^{1,0}, \vec{e}_{h}^{2,0}$, for $t \in [0,T)$,
find $\vec{x}_h( \cdot, t) \in V_h^3, \vec{y}_h( \cdot, t) \in V_{h,0}^3, \vec{\kappa}_h( \cdot, t ) \in V_{h,0}^3 + \vec{\kappa}_b(\cdot,t)$,
$m_h( \cdot, t ) \in V_h, z_h( \cdot, t ), \gamma_h( \cdot, t), p_h( \cdot, t ) \in Q_h$, $\vec{e}^{1}_{h}( \cdot, t ), \vec{e}^{2}_{h}( \cdot, t ) \in V_h^3$ such that, for all $t \in (0,T)$:
\begin{subequations}
\begin{align}
  \label{eq:fem-x}
  \int_0^1 \K \vec{x}_{h,t} \cdot \vec{\phi}_h | \vec{x}_{h,u} |
  - \int_0^1 p_h \vec\tau_h \cdot \vec{\phi}_{h,u}
  \qquad\qquad\qquad\qquad\qquad\qquad & \\
  \nonumber
  - \int_0^1 \bigl( ( \id - \vec\tau_h \otimes \vec\tau_h ) \frac{\vec{y}_{h,u}}{| \vec{x}_{h,u} |}   + z_h \vec\tau_h \times \vec{\kappa}_h \bigr) \cdot \vec{\phi}_{h,u}
  & = 0 \\
\label{eq:fem-y}
  \int_0^1 \Bigl( \bigl(
  \vec{y}_h
  - A ( \vec{\kappa}_h - \alpha^0 \vec{e}^1_h - \beta^0 \vec{e}^2_h )
  \qquad\qquad\qquad\qquad\qquad\qquad & \\
  \nonumber
  - B ( ( \id - \tilde{\vec\tau}_h \otimes \tilde{\vec\tau}_h ) \vec{\kappa}_{h,t} - m_h \tilde{\vec\tau}_h \times \vec{\kappa}_h )
  \bigr) \cdot \vec{\psi}_h \Bigr)_h \abs{ \vec{x}_{h,u} } & = 0 \\
  \label{eq:fem-w}
  \int_0^1 ( \vec{\kappa}_h \cdot \vec\psi_h )_h | \vec{x}_{h,u} | + \frac{\vec{x}_{h,u}}{ | \vec{x}_{h,u} | }  \cdot \vec\psi_{h,u} & = 0
\end{align}
for all $\vec\phi_h \in V_h^3$, $\vec\psi_h \in V_{h,0}^3$,
\begin{align}
  \label{eq:fem-gamma}
  \int_0^1 - ( K^\rot m_h v_h )_h | \vec{x}_{h,u} | - \int_0^1 z_h v_{h,u}
  + \int_0^1 ( \vec{y}_h \cdot ( \tilde{\vec\tau}_h \times \vec{\kappa}_h ) v_h )_h | \vec{x}_{h,u} |
   & = 0, \\
  \label{eq:fem-z}
  \int_0^1 ( z_h - C ( \gamma_h - \gamma^0 ) - D \gamma_{h,t} ) q_h | \vec{x}_{h,u} | & = 0, \\
  \label{eq:fem-m}
  \int_0^1 \gamma_{h,t} q_h | \vec{x}_{h,u} |
  - \int_0^1 m_{u,h} q_h
  + \int_0^1 \vec\tau_h \times \vec{\kappa}_h \cdot \vec{x}_{h,tu} q_h & = 0
\end{align}
for all $q_h \in Q_h$ and $v_h \in V_h$,
\begin{align}
  \label{eq:fem-p}
  \int_0^1 q_h \vec\tau_h \cdot \vec{x}_{h,tu} & = 0,
\end{align}
for all $q_h \in Q_h$,
and
\begin{equation}
  \label{eq:fem-frame}
  \int_0^1 \Bigl( \bigl( \vec{e}_{h,j,t} - \bigl( \tilde{\vec{\tau}}_h \times \tilde{\vec{\tau}}_{h,t} + m_h \tilde{\vec{\tau}}_h \bigr) \times \vec{e}_{h,j} \bigr) \cdot \vec{\phi}_h \Bigr)_h \abs{\vec{x}_{h,u}} = 0, \mbox{ for } j = 1,2,
\end{equation}
for all $\vec{\phi}_h\in V_h^3$,
subject to the initial conditions:
\begin{equation}
  \label{eq:fem-initial}
  \vec{x}_h(\cdot, 0) = \vec{x}_{h}^0, \qquad
  \vec{e}_{h}^{1}(\cdot, 0) = \vec{e}_{h}^{1,0}, \qquad
  \vec{e}_{h}^{2}(\cdot, 0) = \vec{e}_{h}^{1,0},
\end{equation}
and initial equations:
\begin{align}
  \int_0^1 \vec{\kappa_h}(\cdot, 0) \cdot \vec\psi_h | \vec{x}^0_{h,u} |
  + \frac{\vec{x}^0_{h,u}}{ | \vec{x}^0_{h,u} | } \cdot \vec\psi_{h,u}
  & = 0 && \mbox{ for all } \vec\psi_h \in V_{h,0}^3 \\
  \int_0^1 ( \gamma_h(\cdot, 0) - \frac{\vec{e}^{1,0}_{h,u}}{\abs{\vec{x}^0_{h,u}}} \cdot \vec{e}^{2,0}_h ) v_h \abs{\vec{x}^0_{h,u}}
  & = 0 && \mbox{ for all } v_h \in V_h.
\end{align}
\end{subequations}
\end{problem}

\begin{remark}
  Decoupling variables for curvature $\vec{w}_h$ and position $\vec{x}_h$ should be interpreted as a tool for solving the partial differential equation system.
  This is widely used when solving geometric partial differential equations (see e.g. \cite{DziKuwSch02} with the convergence results in \cite{DecDzi08} and the review \cite{DecDziEll05}).
  Since we compute with piecewise linear curves we cannot formulate an exact curvature of the discrete curve.
  This is a key difference between the approach presented here and super-helix and super-clothoid approaches \cite{Bertails2006,Casati2013}.
\end{remark}

Using $\vec{e}^{0}_{h} \equiv \tilde{\vec\tau}_h$, we will see $\{ \vec{e}^{0}_{h}, \vec{e}^{1}_{h}, \vec{e}^{2}_{h} \}$ is a vertex-wise orthonormal frame.
Indeed, we note that \cref{eq:fem-frame} implies we recover the vertex-wise relations
\begin{equation}
  \label{eq:fem-frame-vertex}
  \vec{e}^{j}_{h,t} = \vec\omega_h \times \vec{e}^{j}_{h}, \quad \mbox{ for } j=0,1,2,
\end{equation}
where
$\vec\omega_h = \tilde{\vec\tau}_h \times \tilde{\vec\tau}_{h,t} + m_h \tilde{\vec\tau}_h \in V_h^3$.
This implies the following vertex-wise identities hold:
\begin{equation}
  \label{eq:e-ttauh}
  \vec{e}^{1}_{h} \cdot \vec{e}^{2}_{h} = \vec{e}^{1}_{h} \cdot \tilde{\vec{\tau}}_h = \vec{e}^{2}_{h} \cdot \tilde{\vec{\tau}}_h = 0
  \qquad \mbox{ and } \qquad
  | \tilde{\vec{\tau}}_h | = | \vec{e}^{1}_{h} | = | \vec{e}^{2}_{h} | = 1,
\end{equation}
so long as the initial values satisfy corresponding versions of these identities. In other words we have the $(\tilde{\vec{\tau}}_h, \vec{e}^{1}_{h}, \vec{e}^{2}_{h})$ form an orthonormal frame at each vertex.

  Next, we note that we can write \cref{eq:fem-w} as:
  \[
    \vec{\kappa}_{h}( u_i, \cdot ) = \frac{ \vec{\tau}_h( u_i^+, \cdot ) - \vec{\tau}_h( u_i^-, \cdot ) }{ \frac{1}{2} ( \abs{ \vec{x}_{h,u}( u_i^+ ) } + \abs{ \vec{x}_{h,u}( u_i^-, \cdot ) } ) } \qquad \mbox{ for } i = 2, \ldots, N-1.
  \]
  So that, using the fact that $\abs{ \vec{\tau}_h } = 1$, we can infer that $\vec{\kappa}_h$ and $\tilde{\vec\tau}_h$ are orthogonal at the vertices. Indeed, for all $i= 2, \ldots, N-1$, we have
  \begin{align*}
    \vec{\kappa}_{h}( u_i, \cdot ) \cdot \tilde{\vec\tau}_h( u_i, \cdot )
    & = \frac{ \vec{\tau}_h( u_i^+, \cdot ) - \vec{\tau}_h( u_i^-, \cdot ) }{ \frac{1}{2} ( \abs{ \vec{x}_{h,u}( u_i^+ ) } + \abs{ \vec{x}_{h,u}( u_i^-, \cdot ) } ) } \cdot
    \frac{ \vec{\tau}_{h}( u_i^-, \cdot) + \vec{\tau}_h( u_i^+, \cdot ) }{| \vec{\tau}_{h}( u_i^-, \cdot) + \vec{\tau}_h( u_i^+, \cdot ) |} \\
    & = \frac{ \vec{\tau}_h( u_i^+, \cdot ) \cdot \vec{\tau}_h( u_i^+, \cdot ) - \vec{\tau}_h( u_i^-, \cdot) \cdot \vec{\tau}_h( u_i^-, \cdot ) }{ { \frac{1}{2} ( \abs{ \vec{x}_{h,u}( u_i^+ ) } + \abs{ \vec{x}_{h,u}( u_i^-, \cdot ) } ) }{| \vec{\tau}_{h}( u_i^-, \cdot) + \vec{\tau}_h( u_i^+, \cdot ) |} } \\
    & = 0.
  \end{align*}
  Further, at $i=1$ and $i=N$, the boundary conditions give us that $\vec{\kappa}_h$ and $\tilde{\vec\tau}_h$ are orthogonal directly.
  This implies we can create a decomposition of $\vec{\kappa}_h$, at the vertices, into fields $\alpha_h, \beta_h \in V_h$ given by
  \begin{equation}
    \label{eq:fem-w-decomp}
    \vec{\kappa}_h( u_i, \cdot ) = \alpha_h( u_i, \cdot ) \vec{e}^{1}_{h}( u_i, \cdot ) + \beta_h( u_i, \cdot ) \vec{e}^{2}_{h}( u_i, \cdot ) \qquad \mbox{ for } i = 1, \ldots, N.
  \end{equation}
  Similarly it can be shown that $\vec{y}_h( u_i, \cdot ) \cdot \tilde{\vec\tau}_h( u_i, \cdot ) = 0$ for $i = 1, \ldots, N$.

  \begin{lemma}
  \label{lem:stability}
  If $\alpha^0, \beta^0, \gamma^0$ are independent of time,
  any solution to the above problem satisfies:
  \begin{multline}
    \int_0^1 ( \K \vec{x}_{h,t} \cdot \vec{x}_{h,t} + K^\rot_h m_h^2 ) | \vec{x}_{h,u} | \\
    + \frac{1}{2} \ddt \int_0^1 \bigl( A \bigl( ( \alpha_h - \alpha^0 )^2 + ( \beta_h - \beta^0 )^2 \bigr)_h + C ( \gamma_h - \gamma^0 )^2 \bigr) | \vec{x}_{h,u} | \\
    + \int_0^1 \bigl( ( B ( \alpha_{h,t}^2 + \beta_{h,t}^2 ) )_h + D \gamma_{h,t}^2 \bigr) \abs{ \vec{x}_{h,u}} = 0.
  \end{multline}
\end{lemma}

\begin{proof}
  First we see that \eqref{eq:fem-p} implies that $\abs{ \vec{x}_{h,u} }_t = 0$.

  We test \cref{eq:fem-x} with $\vec{x}_{h,t}$, \cref{eq:fem-gamma} with $-m_h$ and \cref{eq:fem-p} with $p_h$.
  Adding the resulting equations results in
  \begin{multline}
    \label{eq:stab-1}
    \int_0^1 ( \K \vec{x}_{h,t} \cdot \vec{x}_{h,t} + K^\rot m_h^2 ) \abs{ \vec{x}_{h,u} }
    - \int_0^1 \big( ( \id - \vec\tau_h \otimes \vec\tau_h ) \frac{\vec{y}_{h,u}}{ \abs{\vec{x}_{h,u}} } + z_h \vec{\tau}_h \times \vec{\kappa}_h \big) \cdot \vec{x}_{h,tu} \\
    + \int_0^1 z_h m_{h,u}
    - \int_0^1 \big( \vec{y}_{h} \cdot ( \tilde{\vec\tau}_h \times \vec{\kappa}_h ) m_h \big)_h \abs{ \vec{x}_{h,u} }
    = 0.
  \end{multline}

  Next, we take a time derivative of \cref{eq:fem-w} and test the result with $\vec{y}_h$.
  \[
    \int_0^1 ( \vec{\kappa}_{h,t} \cdot \vec{y}_h )_h \abs{ \vec{x}_{h,u} }
    + ( \id - \vec\tau_h \otimes \vec\tau_h ) \frac{ \vec{y}_{h,u} }{ \abs{ \vec{x}_{h,u} } } \cdot \vec{x}_{h,tu} = 0
  \]
  From the result we subtract \cref{eq:fem-y} tested with $\vec{\kappa}_{h,t}$ to see
  \begin{multline}
    \label{eq:stab-2}
    \int_0^1 ( \id - \vec\tau_h \otimes \vec\tau_h ) \frac{ \vec{y}_{h,u} }{ \abs{ \vec{x}_{h,u} } } \cdot \vec{x}_{h,tu}
    + \int_0^1 \Big(
    A ( \vec{\kappa}_h - \alpha^0 \vec{e}^{1}_{h} - \beta^0 \vec{e}^{2}_{h} ) \cdot \vec{\kappa}_{h,t} \Big)_h \abs{ \vec{x}_{h,u} } \\
    + \int_0^1 \Big( B \big(
    ( \id - \tilde{\vec\tau}_h \otimes \tilde{\vec\tau}_h ) \vec{\kappa}_{h,t} - m_h \tilde{\vec\tau}_h \times \vec{\kappa}_h
    \big) \cdot \vec{\kappa}_{h,t}
    \Big)_h \abs{ \vec{x}_{h,u} } = 0.
  \end{multline}
  Using \cref{eq:fem-w-decomp}, the frame equations \cref{eq:fem-frame-vertex} and some simple vector identities gives
  \begin{multline}
    \label{eq:stab-3a}
      \int_0^1 \Big(
      A ( \vec{\kappa}_h - \alpha^0 \vec{e}^{1}_{h} - \beta^0 \vec{e}^{2}_{h} ) \cdot \vec{\kappa}_{h,t} \Big)_h \\
      = \frac{1}{2} \frac{d}{dt} \int_0^1 \Bigl( A \bigl( ( \alpha_h - \alpha^0 )^2 + ( \beta_h - \beta^0 )^2 \bigr) \Bigr)_h | \vec{x}_{h,u} | \\
      + \int_0^1 \bigl( A ( \vec{\kappa}_{h} - \alpha^0 \vec{e}^{1}_{h} - \beta^0 \vec{e}^{2}_{h} ) \cdot \vec{\omega_h} \times \vec{\kappa}_h \bigr)_h \abs{ \vec{x}_{h,u} }.
  \end{multline}
  We can further reduce the right hand side using the definition of $\vec{\omega}_h$ and the fact that $\vec{\kappa}_h \cdot \tilde{\vec\tau}_h = 0$:
  \begin{multline}
    \label{eq:stab-3}
    \int_0^1 \Big( A ( \vec{\kappa}_h - \alpha^0 \vec{e}^{1}_{h} - \beta^0 \vec{e}^{2}_{h} ) \cdot \vec\omega_h \times \vec{\kappa}_h \Big)_h \abs{ \vec{x}_{h,u} } \\
    = \int_0^1 \Big( A ( \vec{\kappa}_h - \alpha^0 \vec{e}^{1}_{h} - \beta^0 \vec{e}^{2}_{h} ) \cdot ( m_h \tilde{\vec\tau}_h ) \times \vec{\kappa}_h\Big)_h \abs{ \vec{x}_{h,u} }.
  \end{multline}
  Similarly, we see that
  \begin{multline}
    \label{eq:stab-4}
    \int_0^1 \bigl( ( B ( \id - \tilde{\vec\tau} \otimes \tilde{\vec\tau} ) \vec{\kappa}_{h,t} - m_h \tilde{\vec\tau}_h ) \cdot \vec{\kappa}_{h,t} \bigr)_h \abs{ \vec{x}_{h,u} } \\
    = \int_0^1 \bigl( B ( \alpha_{h,t}^2 + \beta_{h,t}^2 ) \bigr)_h \abs{ \vec{x}_{h,u} }+ \int_0^1 \bigl( ( B ( \id - \tilde{\vec\tau} \otimes \tilde{\vec\tau} ) \vec{\kappa}_{h,t}- m_h \tilde{\vec\tau}_h ) \cdot ( m \tilde{\vec\tau}_h \times \vec{\kappa}_h ) \bigr)_h \abs{ \vec{x}_{h,u} }.
  \end{multline}
  Combining \cref{eq:stab-3a,eq:stab-3,eq:stab-4} with \cref{eq:stab-2} gives
  \begin{align*}
    & \frac{1}{2} \ddt \int_0^1 \Bigl( A \bigl( ( \alpha_h - \alpha^0 )^2 + ( \beta_h - \beta^0 )^2 \bigr) \Bigr)_h | \vec{x}_{h,u} |
    + \int_0^1 \bigl( B ( \alpha_{h,t}^2 + \beta_{h,t}^2 ) \bigr)_h \abs{ \vec{x}_{h,u} } \\
    & + \int_0^1 ( \id - \vec\tau_h \otimes \vec\tau_h ) \frac{ \vec{y}_{h,u} }{ \abs{ \vec{x}_{h,u} } } \cdot \vec{x}_{h,tu} \\
    & + \int_0^1 \bigl( A ( \vec{\kappa}_{h} - \alpha^0 \vec{e}^{1}_{h} - \beta^0 \vec{e}^{2}_{h} ) \cdot (m_h \tilde{\vec\tau}_h \times \vec{\kappa}_h) \bigr)_h \abs{ \vec{x}_{h,u} } \\
    & + \int_0^1 \bigl( ( B ( \id - \tilde{\vec\tau} \otimes \tilde{\vec\tau} ) \vec{\kappa}_{h,t} ) \cdot ( m \tilde{\vec\tau}_h \times \vec{\kappa}_h ) \bigr)_h \abs{ \vec{x}_{h,u} }
    = 0.
  \end{align*}
  We identify the last two terms of the right-hand side with terms in \cref{eq:fem-y} tested with the test function $\vec\psi_h$ given by
  \[
    \vec\psi_h( u_j ) = \begin{cases}
      0 & \mbox{ for } j = 0, N, \\
      m_h(u_j) \tilde{\vec\tau}_h(u_j) \times \vec{\kappa}_h(u_j) & \mbox{ for } 2 \le j \le N-1.
    \end{cases}
  \]
  Noting that $\vec{y}_h \cdot \vec\psi_h = \vec{y}_h \cdot ( m_h \tilde{\vec\tau}_h \times \vec{\kappa}_h )$, we infer that:
  \begin{multline}
    \label{eq:stab-5}
    \frac{1}{2} \ddt \int_0^1 \Bigl( A \bigl( ( \alpha_h - \alpha^0 )^2 + ( \beta_h - \beta^0 )^2 \bigr) \Bigr)_h | \vec{x}_{h,u} |
    + \int_0^1 \bigl( B ( \alpha_{h,t}^2 + \beta_{h,t}^2 ) \bigr)_h \abs{ \vec{x}_{h,u} } \\
    + \int_0^1 ( \id - \vec\tau_h \otimes \vec\tau_h ) \frac{ \vec{y}_{h,u} }{ \abs{ \vec{x}_{h,u} } } \cdot \vec{x}_{h,tu}
    + \int_0^1 \bigl( \vec{y}_h \cdot (m_h \tilde{\vec\tau}_h \times \vec{\kappa}_h) \bigr)_h \abs{ \vec{x}_{h,u} }
      = 0.
  \end{multline}

  We sum the result of testing \cref{eq:fem-z} with $-\gamma_{h,t}$ and \cref{eq:fem-m} with $z_h$ and rearrange:
  \begin{multline}
    \label{eq:stab-6}
    \frac{1}{2} \ddt \int_0^1 C \abs{ \gamma - \gamma^0 }^2 \abs{ \vec{x}_{h,u} }
    + \int_0^1 D \gamma_{h,t}^2 \abs{ \vec{x}_{h,u} } \\
    - \int_0^1 m_{h,u} z_h
    + \int_0^1 ( \vec\tau_h \times \vec{\kappa}_{h} ) \cdot \vec{x}_{h,tu} z_h = 0.
  \end{multline}

  Adding \cref{eq:stab-2,eq:stab-5,eq:stab-6} gives the desired result.
\end{proof}

\subsection{Fully discrete problem}

To discretize in time we use a uniform partition of the time interval $[0,T]$ into time steps $0 = t_0 < t_1 < \ldots < t_M = T$ where $t_n = n \Delta t$ for $n=0,\ldots,M$.
We denote discrete variables at a time step $t_n$ with a superscript $n$ and the frame vectors will be denoted by $\vec{e}^{1,n}_h$.
Our approach is a first order semi-implicit scheme which results in a linear problem to solve at each time step.
During the numerical experiments, we demonstrate that by choosing to take certain terms implicitly we recover the semi-discrete stability result.
For a variable $\eta$ defined at each time step $0 < n  < M$, we denote the backward difference $\bar\partial \eta^n := ( \eta^n - \eta^{n-1} ) / \Delta t$.
We define $\vec{\kappa}_{b}^n$ by
\[
  \vec{\kappa}_{b}^n := \alpha^0( \cdot, t^n ) \vec{e}^{1,n-1}_{h} + \beta^0( \cdot, t^n ) \vec{e}^{2,n-1}_{h}.
\]

As well as choosing whether to take terms implicitly or explicitly, we also integrate the constraint equation \cref{eq:fem-p} forwards in time and write the frame update equation as an algebraic relation which preserves the nature of the angular velocity vector $\vec{\omega}_h$.

\begin{problem}
  Given preferred curvatures $\alpha^0, \beta^0$, a preferred twist $\gamma^0$, and initial conditions for the parametrization $\vec{x}_h^0$ and the frame $\vec{e}^{1,0}_h, \vec{e}^{2,0}_h$,
  first define $\vec{\kappa}_{h}^0 \in V_{h,0}^3$ and $\gamma_h^0 \in V_h$ as the solutions of
  \begin{align*}
  \int_0^1 \vec{\kappa}_h^0 \cdot \vec\psi_h | \vec{x}^0_{h,u} |
  + \frac{\vec{x}^0_{h,u}}{ | \vec{x}^0_{h,u} | } \cdot \vec\psi_{h,u}
  & = 0 && \mbox{ for all } \vec\psi \in V_{h,0}^3 \\
  \int_0^1 ( \gamma_h^0 - \frac{\vec{e}^{1,0}_{h,u}}{\abs{\vec{x}^0_{h,u}}} \cdot \vec{e}^{2,0}_h ) v_h \abs{\vec{x}^0_{h,u}}
  & = 0 && \mbox{ for all } v_h \in V_h.
  \end{align*}
  Then for $n = 1, \ldots, M$
find $\vec{x}_h^n \in V_h^3, \vec{y}_h^n \in V_{h,0}^3, \vec{\kappa}_h^n \in V_{h,0}^3 + \vec{\kappa}_{b}^n$,
$m_h^n \in V_h, z_h^n, \gamma_h^n, p_h^n \in Q_h$, $\vec{e}^{1,n}_{h}, \vec{e}^{2,n}_{h} \in V_h^3$ such that
\begin{subequations}
  \label{eq:discrete}
\begin{align}
  \label{eq:discrete-x}
  \int_0^1 \K \bar\partial \vec{x}_{h}^n \cdot \vec{\phi}_h | \vec{x}_{h,u}^{n-1} |
  - \int_0^1 p_h^n \vec\tau_h \cdot \vec{\phi}_{h,u}
  \qquad\qquad\qquad\qquad\qquad\qquad \\
  \nonumber
  - \int_0^1 \bigl( ( \id - \vec\tau_h^{n-1} \otimes \vec\tau_h^{n-1} ) \frac{1}{| \vec{x}_{h,u}^{n-1} |} \vec{y}_{h,u}^n  + z_h^n \vec\tau_h^{n-1} \times \vec{\kappa}_h^{n-1} \bigr) \cdot \vec{\phi}_{h,u}
  & = 0 \\
\label{eq:discrete-y}
  \int_0^1 \bigl( \vec{y}_h^{n} - A ( \vec{\kappa}_h^n - \alpha^0( \cdot, t^n ) \vec{e}^{1,n-1}_h - \beta^0( \cdot, t^n ) \vec{e}^{2,n-1}_h )
  \qquad\qquad\qquad\qquad \\
  \nonumber
  - B \bigl( ( \id - \tilde{\vec\tau}_h^{n-1} \otimes \tilde{\vec\tau}_h^{n-1} ) \bar\partial \vec{\kappa}_{h}^n - m_h^{n-1} \tilde{\vec\tau}_h^{n-1} \vec{\kappa}_h^n \bigr)_h \bigr) \cdot \vec\psi_h | \vec{x}_{h,u}^{n-1} | & = 0 \\
  \label{eq:discrete-w}
  \int_0^1 \vec{\kappa}_h^{n} \cdot \vec\psi_h | \vec{x}_{h,u}^{n-1} | + \frac{1}{ | \vec{x}_{h,u}^{n-1} | } \vec{x}_{h,u}^n \cdot \vec\psi_{h,u} & = 0
\end{align}
for all $\vec\phi_h \in V_h^3$, $\vec\psi_h \in V_{h,0}^3$,
\begin{align}
  \label{eq:discrete-gamma}
  \int_0^1 - K^\rot m_h^n v_h | \vec{x}_{h,u}^{n-1} | - \int_0^1 z_h^n v_{h,u}
  + \int_0^1 \vec{y}_h^{n-1} \cdot ( \tilde{\vec\tau}_h^{n-1} \times \vec{\kappa}_h^{n-1} ) v_h | \vec{x}_{h,u}^{n-1} |
   & = 0, \\
  \label{eq:discrete-z}
  \int_0^1 ( z_h^n - C ( \gamma_h^n - \gamma^0( \cdot, t^n)  ) - D \bar\partial \gamma_{h}^n ) q_h | \vec{x}_{h,u}^{n-1} | & = 0, \\
  \label{eq:discrete-m}
  \int_0^1 \bar\partial \gamma_{h}^n q_h | \vec{x}_{h,u}^{n-1} |
  - \int_0^1 m_{u,h} q_h
  + \int_0^1 ( \vec\tau_h^{n-1} \times \vec{\kappa}_h^{n-1} ) \cdot \bar\partial \vec{x}_{h,u}^n q_h & = 0
\end{align}
for all $q_h \in Q_h$ and $v_h \in V_h$,
\begin{align}
  \label{eq:discrete-p}
  \int_0^1 q_h \vec\tau_h^{n-1} \cdot \vec{x}_{h,u}^n & = \int_0^1 | \vec{x}^0_{h,u} | q_h,
\end{align}
for all $q_h \in Q_h$.
Using the abbreviations:
\begin{align}
  \vec{k}_i^n
  & = \tilde{\vec\tau}_h^{n-1}( u_i ) \times \tilde{\vec\tau}_h^n( u_i ),
    &\vec{l}_i^n
  & = \tilde{\vec{\tau}}^n_h( u_i ),
  & \varphi_i^n & = \Delta t \, m_h^n( u_i ),
\end{align}
we apply the Rodrigues formula twice:
\begin{align}
  \label{eq:discrete-e1-R1}
  \tilde{\vec{e}}^{j,n}_{h}( u_i )
  & = \vec{e}^{j,n-1}_{h}( u_i ) ( \tilde{\vec\tau}_h^{n-1}( u_i ) \cdot \tilde{\vec\tau}_h^{n}( u_i ))
    + \vec{k}_i^n \times \vec{e}^{j,n-1}_{h}( u_i ) \\
  \nonumber
  & \qquad + \vec{e}^{j,n-1}_{h}( u_i ) \cdot \vec{k}_i^n
    \vec{k}_i^n
    \frac{1}{1 +  \tilde{\vec\tau}_h^{n-1}( u_i ) \cdot \tilde{\vec\tau}_h^{n}( u_i )} && j=1,2 \\
  \label{eq:discrete-e1-R2}
  {\vec{e}}_{j,h}^n( u_i )
  & = \tilde{\vec{e}}^{j,n}_{h}( u_i ) \cos( \varphi_i )
    + \vec{l}_i^n \times \tilde{\vec{e}}^{j,n}_{h}( u_i ) \sin( \varphi_i ) \\
  \nonumber
  & \qquad + ( \tilde{\vec{e}}^{j,n}_{h}( u_i ) \cdot \vec{l}_i^n )\vec{l}_i^n ( 1 - \cos( \varphi_i ) ) && j=1,2.
\end{align}
\end{subequations}
\end{problem}

The scheme results in a linear system of equations at each time step followed by an algebraic update formula for the frame.
The linear system can be solved using a direct sparse solver.
In this work, the numerical results are computed using the \textsf{UMFPACK} library \cite{umfpack}.

In addition to the usual time discretization, we have chosen to integrate the constraint equation forwards in time.
This gives us more control over the length element $|\vec{x}_{h,u}^k|$ as shown in the following lemma.

\begin{lemma}
  \label{lem:length}
  If there exists a solution such that $\abs{ \vec{\tau}_h^{n} - \vec{\tau}_h^{n-1} }^2 < 2$, then
  \begin{align}
    \abs{ \vec{x}_{h,u}^0 } \le \abs{ \vec{x}_{h,u}^{n} }
    = \frac{ \abs{ \vec{x}_{h,u}^0} }{ 1 - \frac{1}{2} \abs{ \vec{\tau}_h^{n} - \vec{\tau}_h^{n-1} }^2 }.
  \end{align}
\end{lemma}

\begin{proof}
Testing equation \cref{eq:discrete-p} with $q_h = \chi_{[u_i, u_{i+1}]}$, the characteristic function of the interval $[u_i,u_{i-1}]$, gives the element-wise identity
\begin{align*}
  \vec{\tau}_h^{n-1} \cdot \vec{x}_{h,u}^{n} = \abs{ \vec{x}_{h,u}^0 }.
\end{align*}
Then we have
\begin{align*}
  \abs{ \vec{x}_{h,u}^n }
  = \abs{ \vec{x}_{h,u}^n } ( 1 - \vec{\tau}^{n-1}_h \cdot \vec\tau^n_h ) + \abs{ \vec{x}_{h,u}^0 }
  = \frac{1}{2} \abs{ \vec{x}_{h,u}^n } \abs{ \vec{\tau}_h^{n} - \vec{\tau}_{h}^{n-1} }^2 + \abs{ \vec{x}_{h,u}^0 }.
\end{align*}
Since $\frac{1}{2} \abs{ \vec{x}_{h,u}^n } \abs{ \vec{\tau}_h^{n} - \vec{\tau}_{h}^{n-1} }^2\ge 0$, we have $\abs{\vec{x}_{h,u}^n} \ge \abs{ \vec{x}_{h,u}^0 }$.
Furthermore if it holds that $\abs{ \vec{\tau}_h^{n} - \vec{\tau}_h^{n-1} }^2 < 2$, this equation can be rearranged to see the desired result.
\end{proof}

The first rotation \cref{eq:discrete-e1-R1} maps $\tilde{\vec{\tau}}_h^{n-1}$ to $\tilde{\vec\tau}_h^n$ and the second rotation, \cref{eq:discrete-e1-R2}, rotates the frame about the new $\tilde{\vec\tau}_h^n$ (leaving $\tilde{\vec\tau}_h^n$ unaffected). Since we apply the same rotations to the two frame vectors, and these rotations map $\tilde{\vec\tau}_h^{n-1}$ to $\tilde{\vec\tau}_h^n$, this update procedure results preserves the orthogonality of the frame vectors at each vertex.
In practical computations we will see an accumulation of floating point errors (see \cref{fig:model-3d-length-frame-mismatch}, for example). If the errors become too large we may renormalize the frame and continue the computation.

\section{Results}
\label{sec:results}

We provide three test cases for our numerical scheme. In the first we relax a straight rod to one with prescribed curvatures and twist and in the other two we demonstrate the applicability of the method to \textit{C. elegans} locomotion.

\subsection{Relaxation test}

We take as initial configuration of the rod a unit length straight midline curve and constant frame. We then simulate to $T=25$ with
\[
  \alpha^0 = 2 \sin( 3 \pi u / 2 ), \quad \beta^0 = 3 \cos( 3 \pi u / 2 ), \quad \gamma^0 = 5 \cos( 2 \pi u ).
\]
We take material parameters all equal to 1: $L = K_\rot = A = B = C = D = 1$, $\K = \id$. An example of the final configuration is shown in \cref{fig:relaxation-configuration}.
We will use this example to show how the stability result (\cref{lem:stability}) translates to the discrete case.
We also explore the errors in the length element (\cref{lem:length}) and failure to preserve exact orthogonality of the frame.
To show the properties of the scheme, we first simulate with $\Delta t = 1$ and $N=16$ and repeat with the time step $\Delta t$ reduced by a factor of four and doubling $N$: We simulate $\Delta t = 4^{-l}$ and $N=2^{4+l}$ for refinement levels $l = 0,1,\ldots,5$.

\begin{figure}[tbhp]
  \centering
  \includegraphics[width=0.5\textwidth]{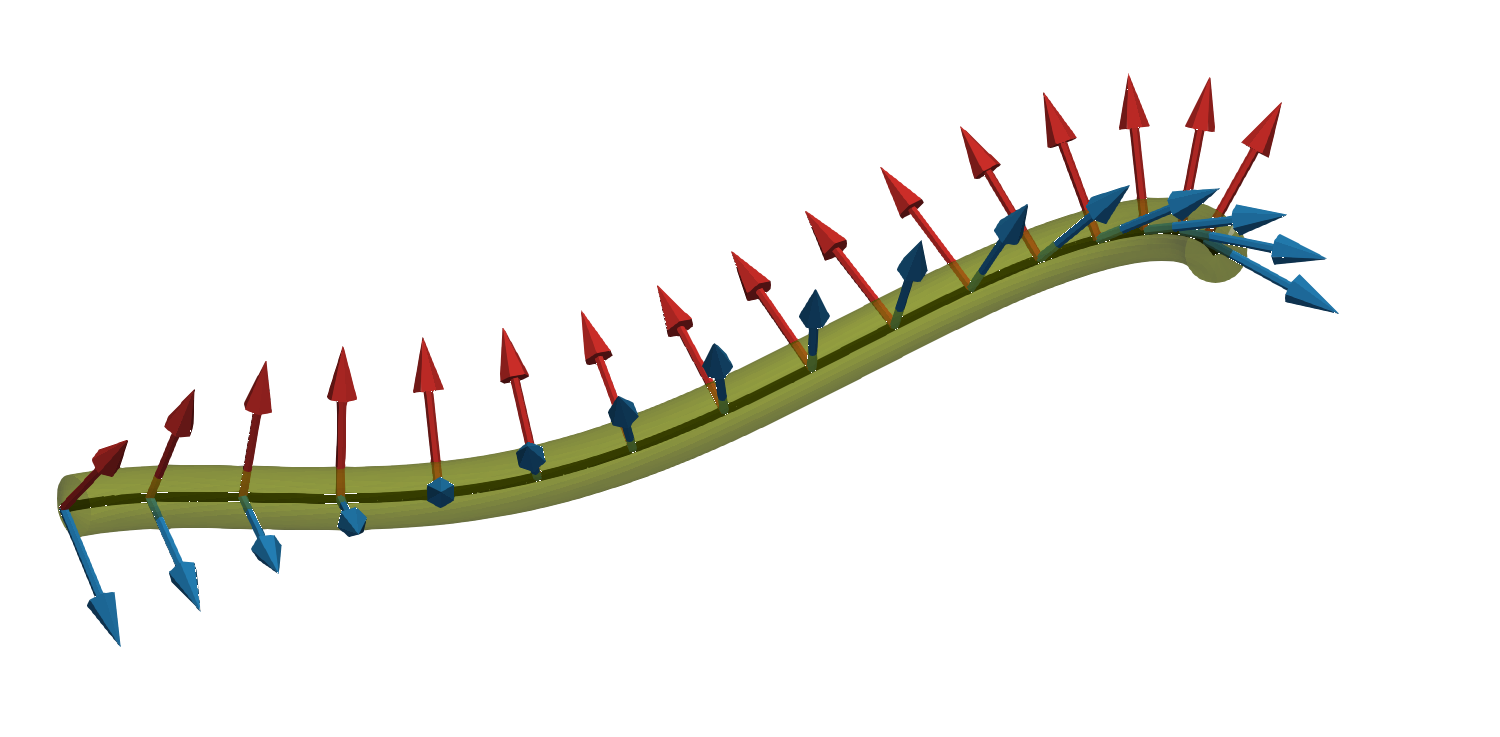}
  \caption{Configuration of the relaxation test with $\Delta t = 10^{-3}$ and $N=128$ at time $T=25$ showing the midline and a sample of frame vectors.
    The colouring is the same as \cref{fig:geometry} except that $\vec{e}^0_h$ is not shown.
  A video of this simulation is presented in \cref{sec:videos}.}
  \label{fig:relaxation-configuration}
\end{figure}

To investigate the fully discrete stability of the scheme, the elastic energy $\mathcal{E}(t^n)$ is shown at each time step $t^n$ in \cref{fig:relaxation-stability} .
We define $\mathcal{E}(t^n)$ by
\begin{align*}
  \mathcal{E}(t^n) := \int_0^1 \Bigl\{ \Bigl( A \abs{ \vec{\kappa}_h - \alpha^0 \vec{e}^{1,n}_h - \beta^0 \vec{e}^{2,n}_h }^2  \Bigr)_h + C( \gamma_h^n - \gamma^0 )^2 \Bigr\} \abs{ \vec{x}_{h,u}^n }
\end{align*}
We see that across all configurations the energy decreases across all our simulation results apart from the very coarsest time steps.

\begin{figure}[tbhp]
  \centering
  \includegraphics{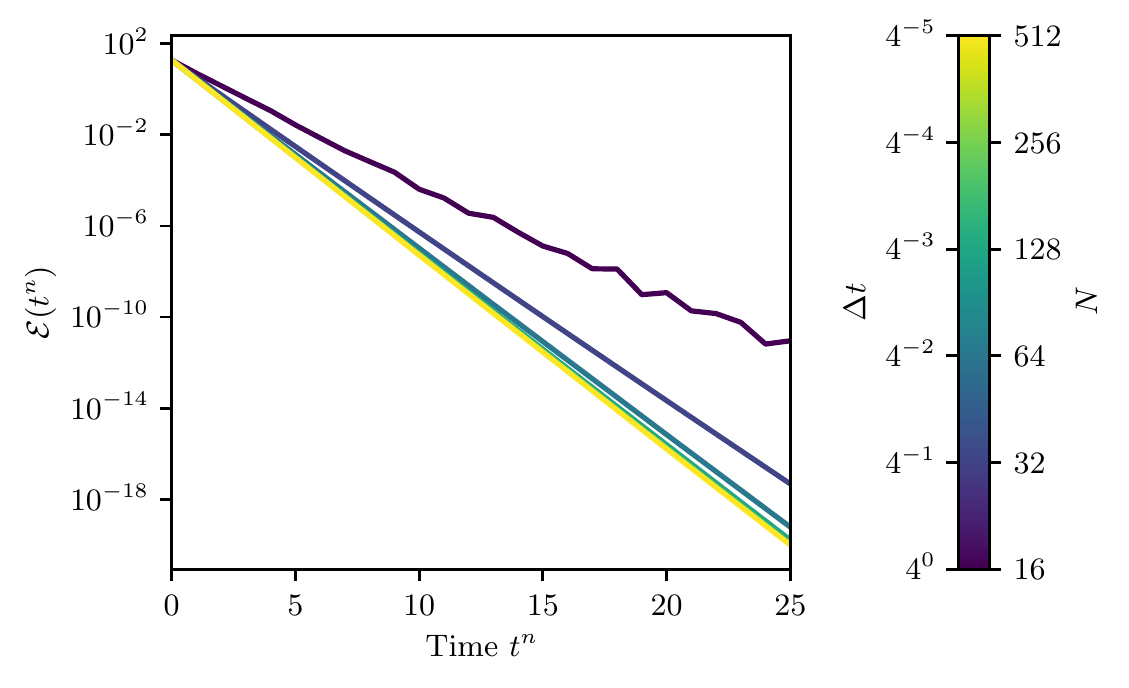}
  \caption{Elastic energy $\mathcal{E}(t^n)$ over time for varying discretization parameters for relaxation test.}
  \label{fig:relaxation-stability}
\end{figure}

Next, we look at the error in the length element.
We follow the refinement procedure detailed above and show the results in \cref{tab:relaxation-length} and \cref{fig:relaxation-length-frame-mismatch}.
The error shown is
\[
  \mathcal{F}_1(t^n) := \abs{ \int_0^1 \abs{ \vec{x}_{h,u}^n } - L }.
\]
The experimental order of convergence, $eoc(\Delta t)$, is
\[
  eoc( \Delta t ) := \log\Bigl( \max_n \mathcal{F}_1(t^n)_{l} / \max_n \mathcal{F}_1(t^n)_{l-1} \Bigr)
  / \log\Bigr( \Delta t_l / \Delta t_{l-1} \Bigr).
\]
We observe that the error in length element decrease in time after the increase from the first initially perfect time step. This matches with the analysis of \cref{lem:length} that the error only depends on the change in tangent from one time step to the next. Since the scheme converges to a stable solution the change in tangent vector reduces in time which results in the reduction of error.
Moreover, we see that the error reduces to second order in the time step which is an order higher than the expected error in the scheme overall.

\begin{figure}[tbhp]
  \centering
  \includegraphics{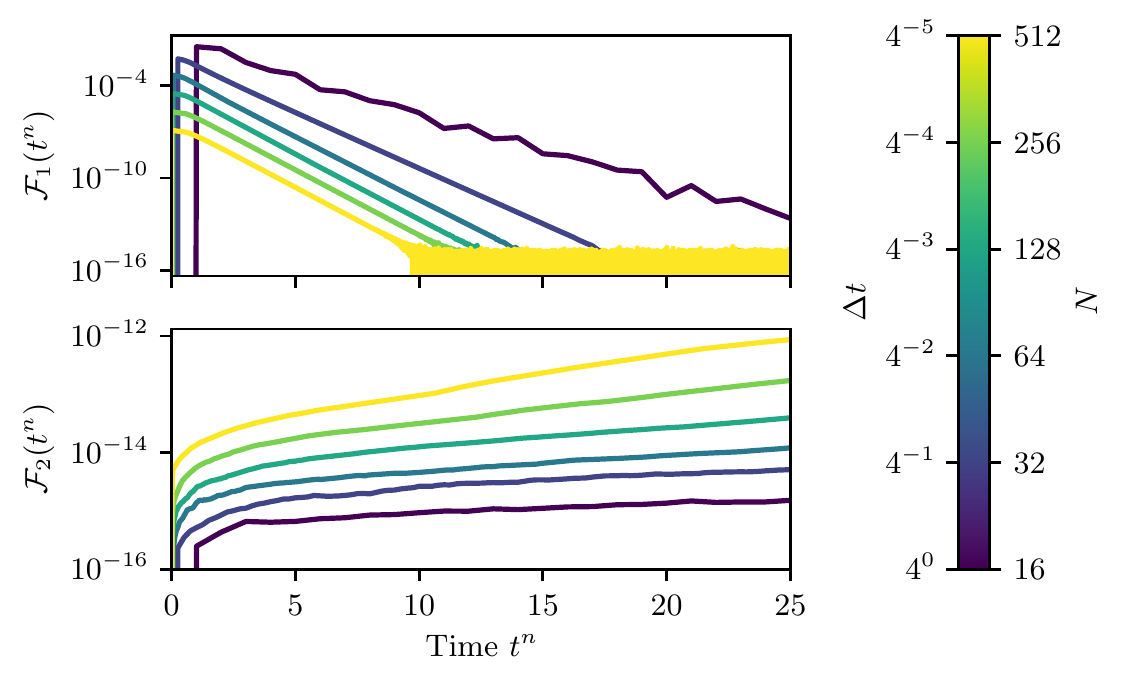}
  \caption{Error of local length constraint $\mathcal{F}_1(t^n)$ and frame orthogonality constraint $\mathcal{F}_2(t^n)$ over time for varying discretization parameters for relaxation test.}
  \label{fig:relaxation-length-frame-mismatch}
\end{figure}

\begin{table}[tbhp]
{\footnotesize
  \centering

  \caption{Maximum error of local length constraint $\mathcal{F}_1(t^n)$ for varying discretization parameters for relaxation test.}
  \label{tab:relaxation-length}

  \clearpage
\begin{tabular}{c|c|c|c}
$\Delta t$ & $N$ & $\max_n \mathcal{F}_1(t^n)$ & $eoc(\Delta t)$ \\
\hline
$1.00000$ & 16 & $3.46788 \cdot 10^{-2}$ & -- \\
$2.50000 \cdot 10^{-1}$ & 32 & $5.64486 \cdot 10^{-3}$ & $1.30952$ \\
$6.25000 \cdot 10^{-2}$ & 64 & $4.89655 \cdot 10^{-4}$ & $1.76355$ \\
$1.56250 \cdot 10^{-2}$ & 128 & $3.34948 \cdot 10^{-5}$ & $1.93488$ \\
$3.90625 \cdot 10^{-3}$ & 256 & $2.14247 \cdot 10^{-6}$ & $1.98330$ \\
$9.76562 \cdot 10^{-4}$ & 512 & $1.34687 \cdot 10^{-7}$ & $1.99580$ \\
\end{tabular}

}
\end{table}

Next, we test for errors in the frame orthogonality conditions with results shown in \cref{tab:relaxation-frame-mismatch} and \cref{fig:relaxation-length-frame-mismatch}.
Here, we look at the maximum over time of the $L^2$-norm of the errors in the orthogonality conditions:
\[
  \mathcal{F}_2(t^n) := \left( \sum_{0 \le j_1 \le j_2 \le 2} \int_0^1 \abs{ \vec{e}^{j_1,n}_h \cdot \vec{e}^{j_2,n}_h - \delta_{j_1, j_2} }^2 \abs{ \vec{x}_{h,u}^n } \right)^{{1}/{2}}.
\]
We observe that the errors are very small across all simulations although the error does increase as we refine in space and time.
Furthermore we see that the errors increase over time which we attribute to accumulation of rounding errors.
The final column in \cref{tab:relaxation-frame-mismatch} shows the maximum change in this error over time. This value is close to machine precision epsilon which indicates that the increase in errors, both in time and as we refine the time step, is due to an increase in the number of time steps.

\begin{table}[tbhp]
{\footnotesize
  \centering

  \caption{Maximum error of frame orthogonality constraint $\mathcal{F}_2(t^n)$ for varying discretization parameters}
  \label{tab:relaxation-frame-mismatch}

  \clearpage
\begin{tabular}{c|c|c|c}
$\Delta t$ & $N$ & $\max_n \mathcal{F}_2(t^n)$ & $\max_n (\mathcal{F}_2(t^n) - \mathcal{F}_2(t^{n-1}))$ \\
\hline
$1.00000$ & 16 & $1.51801 \cdot 10^{-15}$ & $2.46718 \cdot 10^{-16}$ \\
$2.50000 \cdot 10^{-1}$ & 32 & $5.09235 \cdot 10^{-15}$ & $2.31888 \cdot 10^{-16}$ \\
$6.25000 \cdot 10^{-2}$ & 64 & $1.20420 \cdot 10^{-14}$ & $2.11755 \cdot 10^{-16}$ \\
$1.56250 \cdot 10^{-2}$ & 128 & $3.95711 \cdot 10^{-14}$ & $2.26700 \cdot 10^{-16}$ \\
$3.90625 \cdot 10^{-3}$ & 256 & $1.72853 \cdot 10^{-13}$ & $2.35206 \cdot 10^{-16}$ \\
$9.76562 \cdot 10^{-4}$ & 512 & $8.70498 \cdot 10^{-13}$ & $2.02678 \cdot 10^{-16}$ \\
\end{tabular}

}
\end{table}

\subsection{Application to nematode locomotion in two and three spatial dimensions}
\label{sec:application}

We augment the method detailed above by changing the linear drag term to a resistive force term \cite{Keller1976}:
\[
  \K = ( \vec\tau \otimes \vec\tau ) + K ( \id - \vec\tau \otimes \vec\tau ),
\]
which we approximate in the fully discrete scheme by
\[
  \K = ( \vec\tau_h^{n-1} \otimes \vec{\tau}_h^{n-1} ) + K ( \id - \vec\tau_h^{n-1} \otimes \vec\tau_h^{n-1} ).
\]

We demonstrate that we can use the method to simulate \textit{C. elegans} locomotion in two and three dimensions.
We set $L=1$ and restrict our considerations to a stiff environment with $K=40, K_\rot=1$ which should correspond to a crawling behaviour.
We model the \textit{C. elegans} body as an elastic tapered cylinder. We assume that the internal viscous forces do not play a role in the stiff environment \cite{BerBoyTas09} so we consider ($B = D = 0$). For material parameters, we take $A = C = 8 ( ( \varepsilon + u ) ( \varepsilon + 1 - u ) )^{3/2} / ( 1 + 2 \varepsilon )^3$, for $\varepsilon > 0$ small, which corresponds to a uniform elasticity across the shell of a tapered body shape.
We assume that the frame directions correspond to physically meaningful directions within the worm. We assume that $\vec{e}^0$ follows the midline of the body pointing head to tail, $\vec{e}^1$ points in the ventral-dorsal plane - the usual bending direction when considering two dimensional locomotion - and that $\vec{e}^2$ points in the left-right direction.
Muscle contractions generate bending in either the $\vec{e}^1$ or $\vec{e}^2$ directions.

In the usual two dimensional scenario, \textit{C. elegans} generates bending waves in the dorsal-ventral plane. This will be our first test case:
\begin{align*}
  \alpha^0( u, t ) & = ( 10 u + 8 ( 1-u ) ) \sin\left( {2 \pi u}/{0.65} - 0.6 \pi t \right), \\
  \beta^0( u, t ) & = 0, \\
  \gamma^0( u, t ) & = 0.
\end{align*}
We have previously seen \cite{CohRan17-pp} that the first condition can be seen to recreate realistic looking \textit{C. elegans} locomotion postures.
We explore here how well our updated numerical method captures this behaviour. Further results for this test case are shown in \cref{sec:2d-problem}.

It is assumed that \textit{C. elegans} generates undulations in the dorsal-ventral plane due to symmetries in its neural control. However these symmetries do not exist in the head and neck regions (see the discussion in \cite{Bilbao2018}).
Therefore, we propose an alternative control strategy which results in three dimensional, non-planar body configurations. Using the notation $\chi_{[0,1/3]}$ for the characteristic function of the interval $[0,1/3]$, we simulate with
\begin{align*}
  \alpha^0( u, t ) & = ( 10 u + 8 ( 1-u ) ) \sin\left( {2 \pi u}/{0.65} - 0.6 \pi t \right), \\
  \beta^0( u, t ) & = 6 \chi_{[0,1/3]}, \\
  \gamma^0( u, t ) & = 0.
\end{align*}

For both cases, as initial condition we start with an initially straight rod with constant frame and simulate with $\alpha^0( u, 0 ), \beta^0( u, 0 ), \gamma^0( u, 0)$ until $t = 5$ and use the resulting curve as initial condition for our simulation. This means we have an initial condition where the curvatures and twist match the initial preferred curvatures and twist exactly however the frame orthogonality conditions will not hold exactly.

We show some characteristic body positions in \cref{fig:worm-configurations} and head trajectories in \cref{fig:worm-trajectories}. We note that simulations for the two dimensional case we have that the third component is zero and the twist is exactly zero (\cref{fig:worm-2d-kymograms}) whereas there is some twist in the three dimensional scenario even though the preferred twist is zero (\cref{fig:worm-3d-kymograms}).
Further the three dimensional test case demonstrates a non-planer body position and trajectory (see \cref{fig:worm-trajectories,sec:videos}).
The second test case demonstrates a possible strategy for generating three dimensional postures.
We check the errors in the length element and frame orthogonality conditions with results shown in \cref{fig:model-2d-length-frame-mismatch,tab:model-2d-length,tab:model-2d-frame-mismatch} for the two dimensional case and \cref{fig:model-3d-length-frame-mismatch,tab:model-3d-length,tab:model-3d-frame-mismatch} for the three dimensional case.
We observe similar results to the relaxation case.
We see the same second order convergence in the error in the length element although now this error increases and decreases periodically depending on the periodic undulations.
The error is higher overall since the midline continues to move throughout the simulation.
We see that the frame mismatch is again very small across all simulations. This error is initially higher since the initial conditions are derived from simulations.

\begin{figure}[tbhp]
  \centering

  \begin{subfigure}[t]{\textwidth}
    \centering
    \begin{minipage}[t]{0.49\textwidth}
      \centering
      (i)\\
      \includegraphics[width=\textwidth]{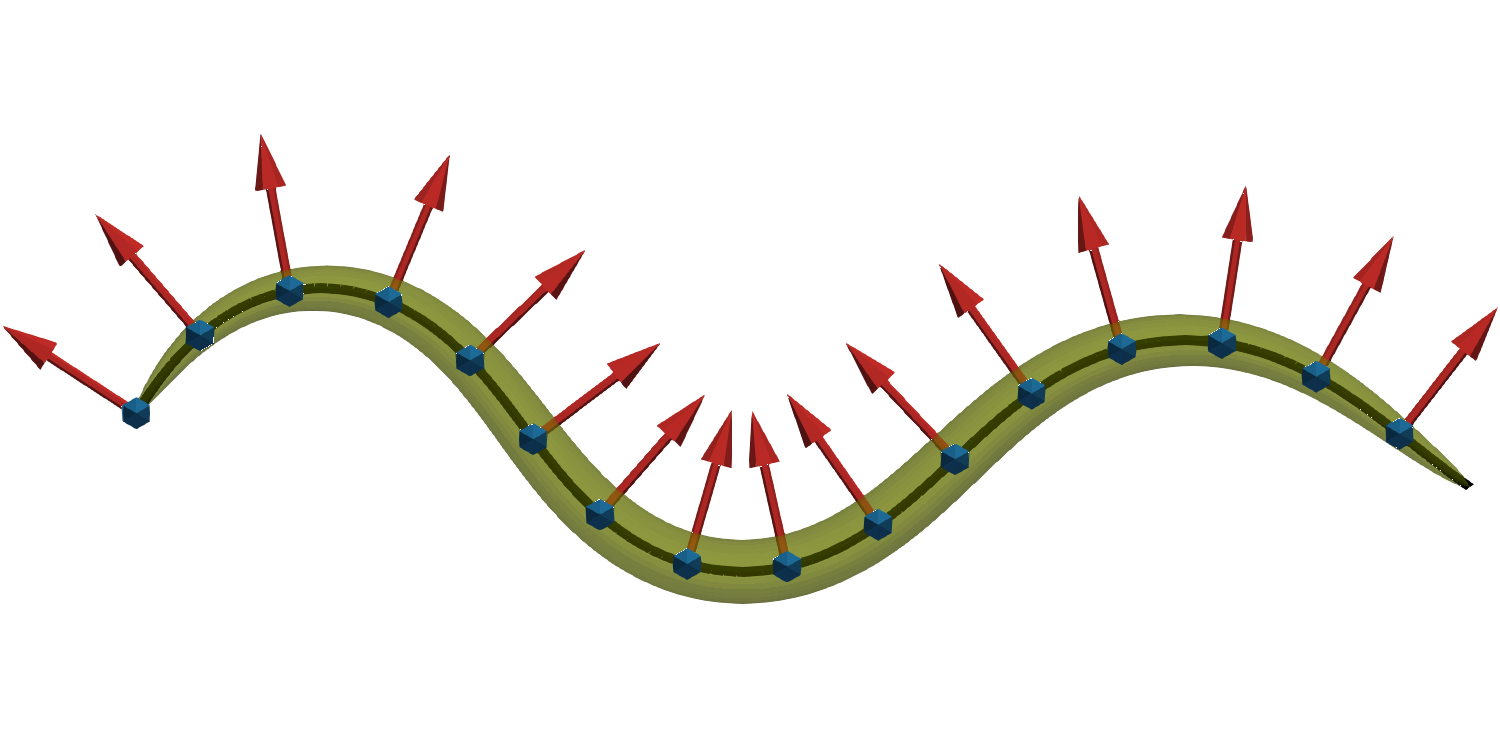}
    \end{minipage}
    \begin{minipage}[t]{0.49\textwidth}
      \centering
      (ii)\\
      \includegraphics[width=\textwidth]{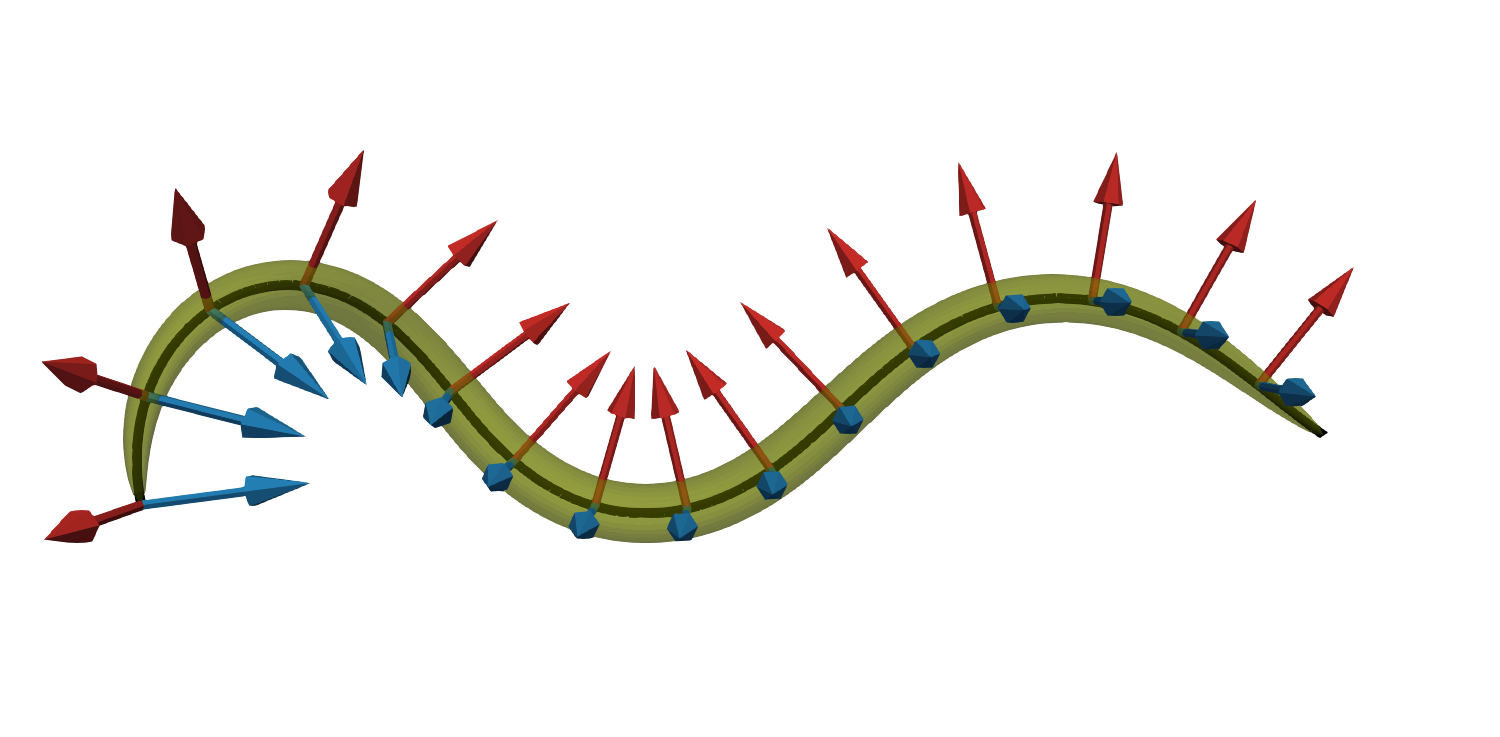}
    \end{minipage}
    \caption{Configurations of (i) two dimensional and (ii) three dimensional \textit{C. elegans} locomotion at time $T=25$. The body is oriented so that $u=0$ is to the left. In this image are shown the midline (black) and frame vectors $\vec{e}^{1,n}_{h}$ (red) $\vec{e}^{2,n}_{h}$ (blue) and the shaded (green) region is the three dimensional volume we are considering as in \cref{fig:geometry}.}
    \label{fig:worm-configurations}
  \end{subfigure}
  \begin{subfigure}[t]{\textwidth}
    \centering
    \includegraphics[width=0.5\textwidth]{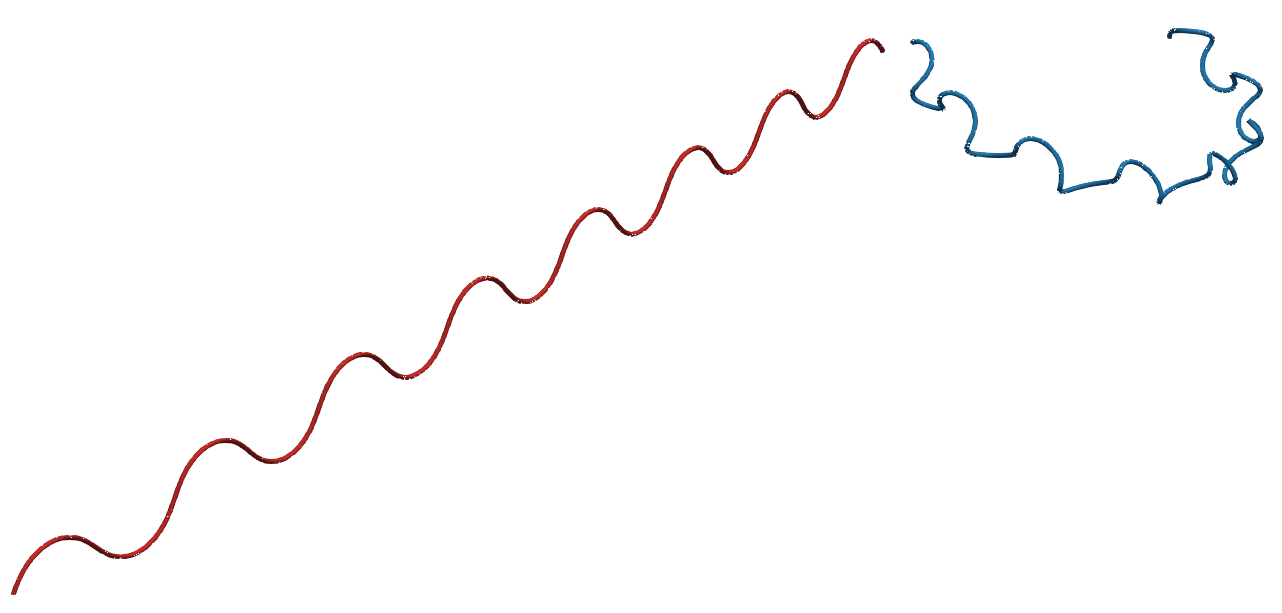}
    \caption{Trajectory of head point $(u=0)$ over time. Red shows the (planar) two dimensional case, blue shows the (non-planar) three dimensional test case.}
    \label{fig:worm-trajectories}
  \end{subfigure}
  \begin{subfigure}[c]{\textwidth}
    \centering
    \includegraphics[width=0.9\textwidth]{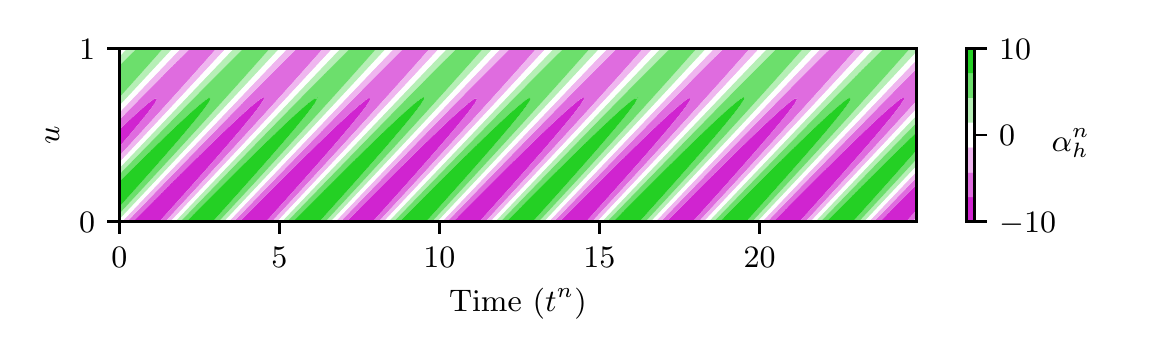}
    \caption{Iso-colour plots of the $\alpha_h^n$ component of generalized curvature for the two dimensional scenario. The other two components $\beta_h^n$ and $\gamma_h^n$ are zero to floating point accuracy.}
    \label{fig:worm-2d-kymograms}
  \end{subfigure} 
  \begin{subfigure}[c]{\textwidth}
    \centering
    \includegraphics[width=0.9\textwidth]{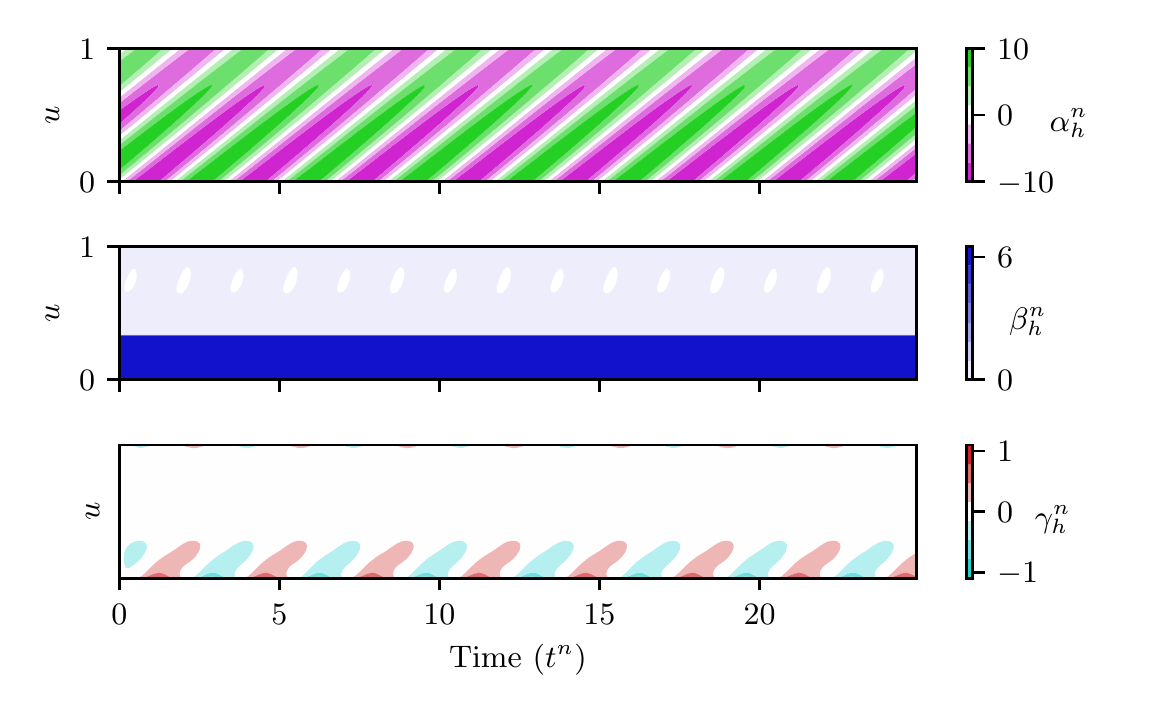}
    \caption{Iso-colour plots of the components of generalized curvature for the three dimensional scenario.}
    \label{fig:worm-3d-kymograms}
  \end{subfigure} 

  \caption{Simulations of \textit{C. elegans} locomotion in two and three dimensions.}
\end{figure}

\begin{figure}[tbhp]
  \centering
  \includegraphics{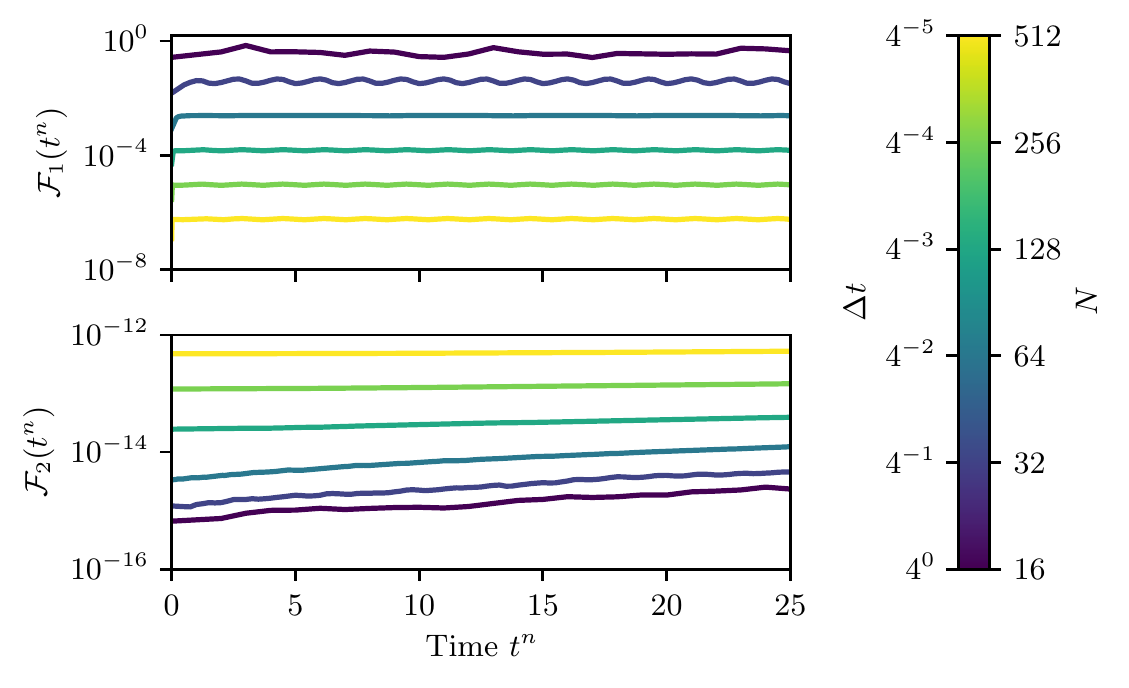}
  \caption{Error of local length constraint $\mathcal{F}_1(t^n)$ and frame orthogonality constraint $\mathcal{F}_2(t^n)$ over time for varying discretization parameters for two dimensional test case.}
  \label{fig:model-2d-length-frame-mismatch}
\end{figure}

\begin{table}[tbhp]
{\footnotesize
  \caption{Results for two dimensional test case.}

  \begin{subfigure}[c]{\textwidth}
    \centering
    \caption{Maximum error of local length constraint $\mathcal{F}_1(t^n)$ for varying discretization parameters.}
    \label{tab:model-2d-length}

    \clearpage
\begin{tabular}{c|c|c|c}
$\Delta t$ & $N$ & $\max_n \mathcal{F}_1(t^n)$ & $eoc(\Delta t)$ \\
\hline
$1.00000$ & 16 & $6.93355 \cdot 10^{-1}$ & -- \\
$2.50000 \cdot 10^{-1}$ & 32 & $4.71797 \cdot 10^{-2}$ & $1.93868$ \\
$6.25000 \cdot 10^{-2}$ & 64 & $2.47289 \cdot 10^{-3}$ & $2.12695$ \\
$1.56250 \cdot 10^{-2}$ & 128 & $1.58417 \cdot 10^{-4}$ & $1.98220$ \\
$3.90625 \cdot 10^{-3}$ & 256 & $9.95399 \cdot 10^{-6}$ & $1.99616$ \\
$9.76562 \cdot 10^{-4}$ & 512 & $6.23034 \cdot 10^{-7}$ & $1.99895$ \\
\end{tabular}

  \end{subfigure}

  \begin{subfigure}[c]{\textwidth}
    \centering
    \caption{Maximum error of frame orthogonality constraint $\mathcal{F}_2(t^n)$ for varying discretization parameters.}
    \label{tab:model-2d-frame-mismatch}

    \clearpage
\begin{tabular}{c|c|c|c}
$\Delta t$ & $N$ & $\max_n \mathcal{F}_2(t^n)$ & $\max_n (\mathcal{F}_2(t^n) - \mathcal{F}_2(t^{n-1}))$ \\
\hline
$1.00000$ & 16 & $2.51950 \cdot 10^{-15}$ & $2.65456 \cdot 10^{-16}$ \\
$2.50000 \cdot 10^{-1}$ & 32 & $4.56873 \cdot 10^{-15}$ & $1.69294 \cdot 10^{-16}$ \\
$6.25000 \cdot 10^{-2}$ & 64 & $1.23614 \cdot 10^{-14}$ & $8.12610 \cdot 10^{-17}$ \\
$1.56250 \cdot 10^{-2}$ & 128 & $3.93250 \cdot 10^{-14}$ & $6.66189 \cdot 10^{-17}$ \\
$3.90625 \cdot 10^{-3}$ & 256 & $1.47282 \cdot 10^{-13}$ & $5.35961 \cdot 10^{-17}$ \\
$9.76562 \cdot 10^{-4}$ & 512 & $5.29361 \cdot 10^{-13}$ & $3.92328 \cdot 10^{-17}$ \\
\end{tabular}

  \end{subfigure}
}
\end{table}

\begin{figure}[tbhp]
  \centering
  \includegraphics{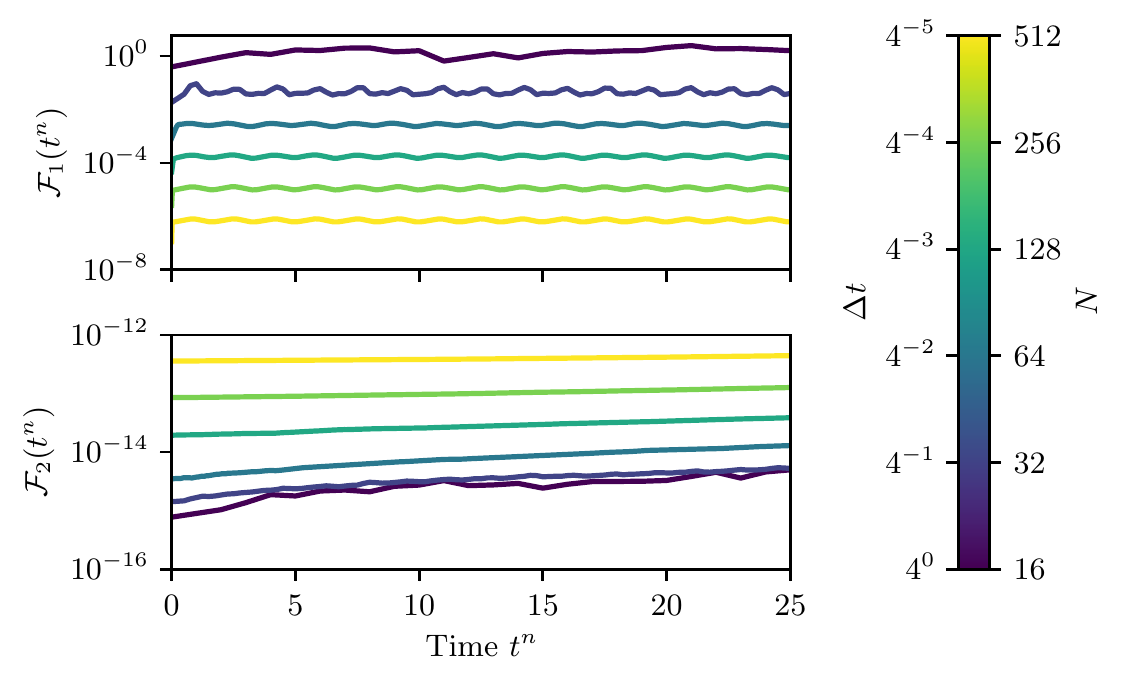}
  \caption{Error of local length constraint $\mathcal{F}_1(t^n)$ and frame orthogonality constraint $\mathcal{F}_2(t^n)$ over time for varying discretization parameters for three dimensional test case.}
  \label{fig:model-3d-length-frame-mismatch}
\end{figure}

\begin{table}[tbhp]
{\footnotesize
  \caption{Results for three dimensional test case.}

  \begin{subfigure}[c]{\textwidth}
    \centering
    \caption{Maximum error of local length constraint $\mathcal{F}_1(t^n)$ for varying discretization parameters.}
    \label{tab:model-3d-length}

    \clearpage
\begin{tabular}{c|c|c|c}
$\Delta t$ & $N$ & $\max_n \mathcal{F}_1(t^n)$ & $eoc(\Delta t)$ \\
\hline
$1.00000$ & 16 & $2.43271$ & -- \\
$2.50000 \cdot 10^{-1}$ & 32 & $9.14325 \cdot 10^{-2}$ & $2.36686$ \\
$6.25000 \cdot 10^{-2}$ & 64 & $3.05806 \cdot 10^{-3}$ & $2.45101$ \\
$1.56250 \cdot 10^{-2}$ & 128 & $1.98811 \cdot 10^{-4}$ & $1.97157$ \\
$3.90625 \cdot 10^{-3}$ & 256 & $1.29005 \cdot 10^{-5}$ & $1.97295$ \\
$9.76562 \cdot 10^{-4}$ & 512 & $8.11140 \cdot 10^{-7}$ & $1.99566$ \\
\end{tabular}

  \end{subfigure}

  \begin{subfigure}[c]{\textwidth}
    \centering
    \caption{Maximum error of frame orthogonality constraint $\mathcal{F}_2(t^n)$ for varying discretization parameters.}
    \label{tab:model-3d-frame-mismatch}

    \clearpage
\begin{tabular}{c|c|c|c}
$\Delta t$ & $N$ & $\max_n \mathcal{F}_2(t^n)$ & $\max_n (\mathcal{F}_2(t^n) - \mathcal{F}_2(t^{n-1}))$ \\
\hline
$1.00000$ & 16 & $4.97680 \cdot 10^{-15}$ & $9.80739 \cdot 10^{-16}$ \\
$2.50000 \cdot 10^{-1}$ & 32 & $5.40303 \cdot 10^{-15}$ & $1.86220 \cdot 10^{-16}$ \\
$6.25000 \cdot 10^{-2}$ & 64 & $1.29219 \cdot 10^{-14}$ & $9.85533 \cdot 10^{-17}$ \\
$1.56250 \cdot 10^{-2}$ & 128 & $3.86360 \cdot 10^{-14}$ & $8.62882 \cdot 10^{-17}$ \\
$3.90625 \cdot 10^{-3}$ & 256 & $1.26807 \cdot 10^{-13}$ & $5.54277 \cdot 10^{-17}$ \\
$9.76562 \cdot 10^{-4}$ & 512 & $4.44560 \cdot 10^{-13}$ & $4.39604 \cdot 10^{-17}$ \\
\end{tabular}

  \end{subfigure}
}
\end{table}

\section{Discussion}
\label{sec:discussion}

We have presented a new finite element scheme for viscoelastic rods suitable for simulating undulatory locomotion in three dimensions.
We have shown analytically that the semi-discrete problem preserves the energy structure of the continuous problem and that the scheme preserves geometric constraints exactly (up to machine precision in practical computations).
Further our numerical experiments show, first, that the analytic properties are realized in a practical fully discrete scheme and that we can capture simple two and three dimensional undulatory behaviours.
Although our method is not the first to tackle this problem, we believe the theoretical grounding and our numerical experiments demonstrate that our method is very well suited for simulations of undulatory locomotion.

We hope in future work to be able to extend the analysis to show stability at the fully discrete level or other convergence properties.
Stability of similar schemes for second order geometric flow curve shortening flow based on parametric representations has been shown at the fully discrete level often by using a non-degenerate re-parameterization of the underlying partial differential equation (e.g.\ \cite{Dzi90,Barrett2011,Elliott2016} and references therein).
The basis of our numerical scheme are the papers \cite{DziKuwSch02,Lin2004} for which there exists numerical analysis of the semi-discrete problem \cite{DecDzi08} but not the fully discrete problem.
There is however further analysis available when alternative parametric schemes are used.
The author of \cite{Bar13} uses a local length constraint for the elastic curve problem in a finite element scheme using B\'ezier curves.
The resulting scheme is unconditionally stable at the discrete level.

In order to apply the method to derive insights for undulatory locomotion, especially for \textit{C.\ elegans}, we must try to understand both material parameters and control mechanisms.
The discussion in \cite{CohRan17-pp} of material parameters for a two-dimensional model still holds true and in three dimensions there are further parameters to explore.
We must also try to understand what gaits are possible and try to find their underlying mechanisms.
For example, the authors of \cite{Bilbao2018} propose a roll manoeuvre in \textit{C.\ elegans} called a ``torsion turn''.
Work is required to understand what control mechanisms would allow such a behaviour.

There are also many details which we have chosen not to consider in our model.
The largest simplification is the use of resistive force theory (RFT).
RFT is a simple first order approximation to the full fluid dynamics which could appear in our model.
The drag coefficients have been experimentally verified for \textit{C.\ elegans} simple forwards locomotion in a variety of media \cite{BerBoyTas09},
although other studies have suggested nonlinear, but still local, drag coefficients give a more accurate representation of the underlying physics \cite{Rabets2014}.
Alternative approaches include using slender body theory \cite{Keller1976,Lighthill1975} which gives a second order non-local correction to RFT or solving the full fluid structure interaction problem using embedded filaments (e.g. \cite{Cortez2018,schoeller2019methods}, a boundary element method \cite{MontenegroJohnson2016} or an immersed boundary method \cite{Lim2012}.
However the RFT approach is crucial in allowing us to solve such a simple one-dimensional system of equations but the model does not hold when non-local interactions become important such as in turning manoeuvres, mating or undulations near walls.
The model and numerical scheme would need to be extended in order to capture these situations.

\section*{Acknowledgements}

The author would like to thank Netta Cohen and Felix Salfelder for discussions which have improved the writing of the manuscript and the computer implementation of the method. I would also like to thank the anonymous reviewers for their careful consideration and helpful comments.

\section*{Declarations of interest}
None

\begin{appendix}

\section{A related two dimensional problem}
\label{sec:2d-problem}

The numerical method presented in the main part of this paper may also be used to understand a related two dimensional scenario.
We consider the problem of a three dimensional rod embedded in a plane without twist - this is the relevant model for nematode locomotion when restricted to two dimensions where the nematode lies on a plate on its left or right side.
This can be used to develop a simpler numerical scheme for these scenarios as we no longer need to solve for the frame and can simply compute the evolution of the midline only.
Indeed we demonstrate numerically that we recover the same results either using the numerical scheme in the main paper or the restricted version here.
The restricted two dimensional scheme here also has the advantage of reducing the numbers of degrees of freedom. In practice this leads to a speed up of simulations by at least a factor of 2 for fine discretizations.

This two dimensional method extends the method previously presented in \cite{CohRan17-pp}.
Here, we include the viscous as well as elastic contribution to the moment.

\subsection{The model}

We consider the same model as in \cref{sec:model} except now we impose that the parameterization of the midline lies in the plane $\vec{x} \in \R^2 \times \{0\}$ and the frame has zero twist ($\gamma \equiv 0$).
For the orthonormal frame, we can use $\vec{e}^{0} \equiv \vec\tau, \vec{e}^{1} \equiv \vec\nu := \vec\tau^\perp, \vec{e}^{2} := ( 0, 0, 1 )$, where $(\cdot)^\perp$ denotes rotation by $\pi/2$ about $\vec{e}^{2}$.
This orthonormal frame has zero twist ($\gamma \equiv 0$) and we have zero curvature in the $\vec{e}_2$ direction $(\beta \equiv 0)$.
Therefore we can derive an appropriate model by replacing \cref{eq:moment} by
\begin{align*}
  \vec{M} = \vec\tau \times \Bigl\{ A \bigl( ( \alpha - \alpha^0 ) \vec\nu \bigr)
  + B \alpha_t \vec\nu \Bigr\}
  = \big\{ A ( \alpha - \alpha^0 ) + B \alpha_t \} \vec{e}_2.
\end{align*}
More compactly, we can write
\[
  \vec{M} = \vec\tau \times \vec{y}, \qquad \vec{y} =  \big\{ A ( \alpha - \alpha^0 ) + B \alpha_t \} \vec{e}_2.
\]
Then \cref{eq:model} may be replaced by a two dimensional version:
\begin{subequations}
  \label{eq:model-2d}
  \begin{align}
    \label{eq:model-2d-x}
    \K \vec{x}_t + \frac{1}{| \vec{x}_u |} \bigl( p \vec\tau \bigr)_u
    + \frac{1}{|\vec{x}_u|} ( \id - \vec\tau \otimes \vec\tau) \frac{ \vec{y}_u }{| \vec{x}_u |} & = \vec{0}
    \\
    \label{eq:model-2d-p}
    \vec\tau \cdot \vec{x}_{tu} & = 0.
  \end{align}
  Here $\id$ is now the $2 \times 2$ identity matrix and $\otimes$ is the outer product given by $( \vec{a} \otimes \vec{b} )_{ij} = \vec{a}_i \vec{b}_j$ for $i,j=1,2$, $\vec{a}, \vec{b} \in \R^2$.
  For boundary conditions we assume zero force and zero moment at $u=0,1$. That is
  \begin{align}
    \label{eq:model-2d-bc0}
    p \vec\tau
    + ( \id - \vec\tau \otimes \vec\tau) \frac{\vec{y}_u}{| \vec{x}_u |} & = \vec{0}
    && \mbox{ for } u = 0,1 \\
    \label{eq:model-2d-bc1}
    \vec\tau \times \vec{y} & = \vec{0}
                              && \mbox{ for } u = 0,1.
  \end{align}
\end{subequations}

\subsection{Weak form}

The weak form of our two dimensional problem is:

\begin{problem}
  Given a preferred curvature $\alpha^0$ and an initial condition for the parameterization $\vec{x}^0$,
  find $\vec{x} , \vec{y}, \vec{\kappa} \colon [0,1] \times [0,T] \to \R^2$ (with the conditions \cref{eq:model-bc0,eq:model-bc1,eq:model-w-b} at the boundaries), $p \colon [0,1] \times [0,T] \to \R$,
  such that, for almost every $t \in (0,T)$:
  \begin{subequations}
    \label{eq:weak-2d}
    \begin{align}
      \label{eq:weak-2d-x}
      \int_0^1 \K \vec{x}_t \cdot \vec{\phi} | \vec{x}_u |
      - \int_0^1 p \vec\tau \cdot \vec{\phi}_u
      - \int_0^1 \bigl( ( \id - \vec\tau \otimes \vec\tau ) \frac{1}{| \vec{x}_u |} \vec{y}_u  \bigr) \cdot \vec{\phi}_u
      & = 0 \\
      \label{eq:weak-2d-y}
      \int_0^1 \bigl( \vec{y} - A ( \vec{\kappa} - \alpha^0 \vec{\nu} )
      - B ( \id - {\vec\tau} \otimes {\vec\tau} ) \vec{\kappa}_{t}
      \bigr) \abs{ \vec{x}_{u} }
      & = 0 \\
      \label{eq:weak-2d-w}
      \int_0^1 \vec{\kappa} \cdot \vec\psi | \vec{x}_u | + \frac{\vec{x}_u}{ | \vec{x}_u | } \cdot \vec\psi_u & = 0
    \end{align}
    for all $\vec\phi \in V^2, \vec\psi \in V_0^2$,
    \begin{align}
      \label{eq:weak-2d-p}
      \int_0^1 q \vec\tau \cdot \vec{x}_{tu} & = 0,
    \end{align}
    for all $q \in Q$,
  \end{subequations}
subject to the initial condition
\begin{equation}
  \vec{x}( \cdot, 0) = \vec{x}^0, \qquad
\end{equation}
and initial equation
\begin{align}
  \int_0^1 \vec{\kappa}(\cdot, 0) \cdot \vec\psi | \vec{x}^0_u |
  + \frac{\vec{x}^0_u}{ | \vec{x}^0_u | } \cdot \vec\psi_u
  & = 0 && \mbox{ for all } \vec\psi \in V_0^2.
\end{align}
\end{problem}

\subsection{Numerical method}

\subsubsection{Semi-implicit scheme}

Our approach follows the same steps as \cref{sec:method}. We again use the spaces of piecewise linear and piecewise constant spaces $V_h$ and $Q_h$ and the two approximations of the tangent vector \cref{eq:tauh-1,eq:tauh-2}. We use the second of these definitions to define an approximation to the normal vector $\vec\nu_h := \tilde{\vec\tau}_h^\perp$.
We make the same choice of approximation spaces as before and arrive at the scheme:

\begin{problem}
  Given a preferred curvature $\alpha^0$, and an initial condition for parameterization $\vec{x}_h$, for $t \in [0,T]$,
  find $\vec{x}_h( \cdot, t) \in V_h^2, \vec{y}_h( \cdot, t) \in V_{h,0}^2, \vec{\kappa}_h( \cdot, t ) \in V_{h,0}^2 + \vec{\kappa}_b(\cdot,t)$,
  $p_h( \cdot, t ) \in Q_h$, $\vec{\nu}_{h}( \cdot, t ) = \tilde{\vec\tau}_h^\perp \in V_h^2$ such that, for all $t \in (0,T)$:
  \begin{subequations}
    \begin{align}
      \label{eq:fem-2d-x}
      \int_0^1 \K \vec{x}_{h,t} \cdot \vec{\phi}_h | \vec{x}_{h,u} |
      - \int_0^1 p_h \vec\tau_h \cdot \vec{\phi}_{h,u}
      \qquad\qquad\qquad\qquad\qquad & \\
      \nonumber
      - \int_0^1 \bigl( ( \id - \vec\tau_h \otimes \vec\tau_h ) \frac{\vec{y}_{h,u}}{| \vec{x}_{h,u} |}  \bigr) \cdot \vec{\phi}_{h,u}
      & = 0 \\
    \label{eq:fem-2d-y}
      \int_0^1 \Bigl( ( \vec{y}_h \cdot \vec\psi_h )_h
      - ( A \vec{\kappa}_h \cdot \vec\psi_h )_h
      + \bigl( A \alpha^0 \vec{\nu}_{h} \cdot \vec{\psi}_h \bigr)_h
      \qquad\qquad\quad & \\
      \nonumber
      - \left(
      B ( \id - \tilde{\vec\tau}_h \otimes \tilde{\vec\tau}_h ) \vec{\kappa}_{h,t}
      \right)_h
      \Bigr) \abs{ \vec{x}_{h,u} } & = 0 \\
      \label{eq:fem-2d-w}
      \int_0^1 \vec{\kappa}_h \cdot \vec\psi_h | \vec{x}_{h,u} | + \frac{\vec{x}_{h,u}}{ | \vec{x}_{h,u} | }  \cdot \vec\psi_{h,u} & = 0
    \end{align}
    for all $\vec\phi_h \in V_h^2$, $\vec\psi_h \in V_{h,0}^2$,
    \begin{align}
      \label{eq:fem-2d-p}
      \int_0^1 q_h \vec\tau_h \cdot \vec{x}_{h,tu} & = 0,
    \end{align}
    for all $q_h \in Q_h$,
  \end{subequations}
subject to the initial condition:
\begin{equation}
  \label{eq:fem2d-initial}
  \vec{x}_h(\cdot, 0) = \vec{x}_{h}^0,
\end{equation}
and initial equation:
\begin{align}
  \int_0^1 \vec{\kappa_h}(\cdot, 0) \cdot \vec\psi_h | \vec{x}^0_{h,u} |
  + \frac{\vec{x}^0_{h,u}}{ | \vec{x}^0_{h,u} | } \cdot \vec\psi_{h,u}
  & = 0 && \mbox{ for all } \vec\psi_h \in V_{h,0}^2.
 \end{align}
\end{problem}

The stability result can be shown in the same way. We do not show the proof here.

\begin{lemma}
  \label{lem:stability-2d}
  If $\alpha^0$ is independent of time,
  any solution to the above problem satisfies:
  \begin{multline*}
    \int_0^1 \K \vec{x}_{h,t} \cdot \vec{x}_{h,t} | \vec{x}_{h,u} |
    + \frac{1}{2} \ddt \int_0^1 \bigl( A ( \alpha_h - \alpha^0 )^2 \bigr)_h | \vec{x}_{h,u} |
    + \int_0^1 ( B \alpha_{h,t}^2 )_h \abs{ \vec{x}_{h,u}} = 0.
  \end{multline*}
\end{lemma}

\subsubsection{Fully discrete scheme}

\begin{problem}
  Given a preferred curvature $\alpha^0$, and an initial condition for the parameterization $\vec{x}_h^0$,
  first define $\vec{\kappa}_{h}^0 \in V_{h,0}^2$ and as the solutions of
  \begin{align*}
  \int_0^1 \vec{\kappa}_h^0 \cdot \vec\psi_h | \vec{x}^0_{h,u} |
  + \frac{\vec{x}^0_{h,u}}{ | \vec{x}^0_{h,u} | } \cdot \vec\psi_{h,u}
  & = 0 && \mbox{ for all } \vec\psi \in V_{h,0}^2.
  \end{align*}
  Then,
  for $n = 1, \ldots, M$,
  find $\vec{x}_h^n \in V_h^2, \vec{y}_h^n \in V_{h,0}^2, \vec{\kappa}_h^n \in V_{h,0}^2 + \vec{\kappa}_{b}^n$,
  $p_h^n \in Q_h$ such that
  \begin{subequations}
    \label{eq:discrete-2d}
    \begin{multline}
      \label{eq:discrete-2d-x}
      \int_0^1 \K \bar\partial \vec{x}_{h}^n \cdot \vec{\phi}_h | \vec{x}_{h,u}^{n-1} |
      - \int_0^1 p_h^n \vec\tau_h \cdot \vec{\phi}_{h,u} \\
      - \int_0^1 ( \id - \vec\tau_h^{n-1} \otimes \vec\tau_h^{n-1} ) \frac{1}{| \vec{x}_{h,u}^{n-1} |} \vec{y}_{h,u}^n  \cdot \vec{\phi}_{h,u}
      = 0,
    \end{multline}
    for all $\vec\phi_h \in V_h^2$,
    \begin{align}
      \label{eq:discrete-2d-y}
      \int_0^1 \bigl( \vec{y}_h^{n} - A ( \vec{\kappa}_h^n - \alpha^0( \cdot, t^n ) \vec{\nu}_{h}^{n-1} ) \qquad\qquad\qquad \\
      \nonumber
      - B \bigl( ( \id - \tilde{\vec\tau}_h^{n-1} \otimes \tilde{\vec\tau}_h^{n-1} ) \bar\partial \vec{\kappa}_{h}^n \bigr)_h \bigr) \cdot \vec\psi_h | \vec{x}_{h,u}^{n-1} | & = 0 \\
      \label{eq:discrete-2d-w}
      \int_0^1 \vec{\kappa}_h^{n} \cdot \vec\psi_h | \vec{x}_{h,u}^{n-1} | + \frac{1}{ | \vec{x}_{h,u}^{n-1} | } \vec{x}_{h,u}^n \cdot \vec\psi_{h,u} & = 0
    \end{align}
    for all $\vec\psi_h \in V_{h,0}^2$,
    \begin{align}
      \label{eq:discrete-2d-p}
      \int_0^1 q_h \vec\tau_h^{n-1} \cdot \vec{x}_{h,u}^n & = \int_0^1 | \vec{x}_{h,0,u} | q_h,
    \end{align}
    for all $q_h \in Q_h$.
  \end{subequations}
\end{problem}

We see that the treatment of the length constraint has not changed so we immediately recover the result of \cref{lem:length}.
Furthermore, we can find the frame without solving any more equations so that we preserve the frame orthogonality exactly.

\subsection{Numerical results}

We demonstrate the similarities between the two and three dimensional versions of the models by simulating the two dimensional test case from the main paper. We use the same refinement strategy as the main paper also.
First we compute the centre of mass at the final time step for the two dimensional scheme \cref{eq:discrete-2d} and the three dimensional scheme \cref{eq:discrete} and show the norm of the difference in \cref{tab:2d3d-com-change}.
We see that the error at the final time is very small and is approximately machine precision epsilon per time step.
This demonstrates that the schemes are exactly the same up to rounding errors.
Second, we compare the how long each simulation takes to run taking the average run time and range across ten simulations. This includes the 5 time unit initial simulation in order to generate the initial condition.
We see that for reasonable levels of resolution the two dimensional version of the code runs in less than half the time of the three dimensional version.

\begin{table}[tbhp]
{\footnotesize
  \centering
  \caption{Comparison of final centre of mass}
  \label{tab:2d3d-com-change}

  \clearpage
\begin{tabular}{c|c|c|c}
$\Delta t$ & $N$ & Difference & Difference per time step \\
\hline
$1.00000$ & $16$ & $1.43005 \cdot 10^{-14}$ & $5.72018 \cdot 10^{-16}$ \\
$2.50000 \cdot 10^{-1}$ & $32$ & $1.57213 \cdot 10^{-14}$ & $1.57213 \cdot 10^{-16}$ \\
$6.25000 \cdot 10^{-2}$ & $64$ & $3.45280 \cdot 10^{-14}$ & $8.63199 \cdot 10^{-17}$ \\
$1.56250 \cdot 10^{-2}$ & $128$ & $6.04952 \cdot 10^{-13}$ & $3.78095 \cdot 10^{-16}$ \\
$3.90625 \cdot 10^{-3}$ & $256$ & $1.87147 \cdot 10^{-12}$ & $2.92417 \cdot 10^{-16}$ \\
$9.76562 \cdot 10^{-4}$ & $512$ & $4.87944 \cdot 10^{-13}$ & $1.90603 \cdot 10^{-17}$ \\
\end{tabular}
}
\end{table}
\begin{table}[tbhp]
{\footnotesize
  \centering
  \caption{Comparison of timings}
  
  \clearpage
\begin{tabular}{c|c|c|c|c}
$\Delta t$ & $N$ & Time (2d) & Time (3d) & Ratio \\
\hline
$1.00000$ & $16$ & $9.09700 \cdot 10^{-1}$ $\pm$ $1.87\%$ & $9.25140 \cdot 10^{-1}$ $\pm$ $8.90\%$ & $1.01697$ \\
$2.50000 \cdot 10^{-1}$ & $32$ & $9.43000 \cdot 10^{-1}$ $\pm$ $2.94\%$ & $9.89800 \cdot 10^{-1}$ $\pm$ $3.80\%$ & $1.04963$ \\
$6.25000 \cdot 10^{-2}$ & $64$ & $1.30110$ $\pm$ $3.09\%$ & $1.74390$ $\pm$ $3.24\%$ & $1.34033$ \\
$1.56250 \cdot 10^{-2}$ & $128$ & $3.67120$ $\pm$ $5.04\%$ & $8.17420$ $\pm$ $1.07\%$ & $2.22657$ \\
$3.90625 \cdot 10^{-3}$ & $256$ & $2.21780 \cdot 10^{1}$ $\pm$ $1.20\%$ & $5.15220 \cdot 10^{1}$ $\pm$ $3.20\%$ & $2.32311$ \\
$9.76562 \cdot 10^{-4}$ & $512$ & $1.63716 \cdot 10^{2}$ $\pm$ $6.29\%$ & $3.97100 \cdot 10^{2}$ $\pm$ $3.10\%$ & $2.42554$ \\
\end{tabular}
}
\end{table}

\section{Simulation videos}
\label{sec:videos}

We also attach three videos of simulations presented in the main paper:

\begin{description}
  \item[{Video 1.}] Conformation of the rod during the relaxation test with $\Delta t = 10^{-3}$ and $N=128$ for times in $[0,25]$.
    The colouring is the same as \cref{fig:relaxation-configuration}.
    \url{https://www.youtube.com/watch?v=aih4-yPn0lk}

  \item[{Video 2.}] Conformation of the body during the two dimensional locomotion test with $\Delta t = 10^{-3}$ and $N=128$ for times in $[0,25]$.
    The colouring is the same as \cref{fig:worm-configurations}(i).
    \url{https://www.youtube.com/watch?v=mWtLNb93RGY}

  \item[{Video 3.}] Conformation of the body during the three dimensional locomotion test with $\Delta t = 10^{-3}$ and $N=128$ for times in $[0,25]$.
    The colouring is the same as \cref{fig:worm-configurations}(ii).
    \url{https://www.youtube.com/watch?v=FO8opOIykLs}
  \end{description}

\end{appendix}

\end{document}